\documentclass[11pt,letterpaper]{article}
\usepackage[T1]{fontenc}
\usepackage{lmodern}
\usepackage[margin=1in]{geometry}
\usepackage{microtype}
\usepackage{amsmath,amssymb,amsthm}
\usepackage{graphicx,stmaryrd}
\usepackage{threeparttable,booktabs,multirow,makecell}
\usepackage{hyperref,cleveref}
\hypersetup{hidelinks,
  pdftitle={Error Analysis for Solving Elliptic Interface Problems with Randomized Neural Networks},
  pdfauthor={Jun Hu and Sidi Wu}}
\allowdisplaybreaks

\newtheorem{theorem}{Theorem}[section]
\newtheorem{lemma}[theorem]{Lemma}
\theoremstyle{definition}
\newtheorem{definition}[theorem]{Definition}
\newtheorem{assumption}{Assumption}[section]
\theoremstyle{remark}
\newtheorem{remark}[theorem]{Remark}
\crefname{assumption}{Assumption}{Assumptions}

\newcommand{\norm}[1]{\left\lVert#1\right\rVert}
\newcommand{\R}{\mathbb{R}}
\newcommand{\x}{\mathbf{x}}
\newcommand{\F}{\mathcal{F}}
\renewcommand{\vec}[1]{\boldsymbol{#1}}
\newcommand{\bE}{\mathbb{E}}
\newcommand{\mP}{\mathcal{P}}
\newcommand{\mM}{\mathcal{M}}
\newcommand{\mS}{\mathcal{S}}
\newcommand{\mN}{\mathcal{N}}
\newcommand{\generalfunc}{v}
\newcommand{\generalNN}{u}
\newcommand{\solufunc}{v^*}
\newcommand{\coeff}{\alpha}

\title{Error Analysis for Solving Elliptic Interface Problems with Randomized Neural Networks}
\author{Jun Hu\thanks{School of Mathematical Sciences, Peking University, Beijing 100871, China, and Chongqing Research Institute of Big Data, Peking University, Chongqing 401332, China. Email: \texttt{hujun@math.pku.edu.cn}.}
\and Sidi Wu\thanks{Department of Mathematics, Faculty of Science, China University of Mining and Technology, Beijing 100083, China. Email: \texttt{wsd@cumtb.edu.cn}.}}
\date{}

\begin{document}
\maketitle

\begin{abstract}
This paper presents a unified error analysis of randomized neural networks (RaNNs) for elliptic interface problems. We first derive an integral representation of Barron functions using the standard $\tanh$ activation, which yields a $C^2$-approximation bound valid for any hidden-parameter distribution supported on a cube with strictly positive density. We then establish a quantitative $\mathcal{O}(n^{-1})$ generalization bound for the composite loss that enforces the partial differential equation, boundary, and interface conditions simultaneously, where $n$ denotes the number of training samples. Combining these results with the optimization guarantees for RaNNs, we obtain an a priori $L^2$ error estimate for the neural network solution in the presence of interface discontinuities. Numerical experiments are presented to validate the theoretical results.
\end{abstract}



\section{Introduction}
Elliptic interface problems are widespread in science and engineering, with applications spanning fluid mechanics, materials science, and biomolecular modeling. Accurately solving such problems requires numerical schemes capable of resolving discontinuities and enforcing interface conditions. Mesh-based methods have proven effective but rely on high-quality body-fitted meshes or specialized treatments of mesh-interface intersections, which can become prohibitively expensive for complicated interfaces \cite{chen2023arbitrarily, guittet2015solving,leveque1994immersed,lu2008recent}. These difficulties have motivated the exploration of meshfree, learning-based approaches, among which neural networks offer a promising alternative. Recent efforts have proposed network architectures tailored to interface discontinuities and demonstrated promising empirical performance \cite{he2022mesh,lai2025hard,liu2020MscaleDNN,wang2020mesh,wu2022interfaced,zhu2023physics}. Theoretically, a complete analysis of neural network solvers typically requires accounting for the consistency of the loss function, the approximation capability of the network, the generalization (or statistical) error arising from finite samples, and the optimization error incurred during training. However, most existing theories focus on partial differential equations (PDEs) with smooth coefficients \cite{jiao2025drm,lu2022machine,mishra2023estimates,shin2020on}. When extended to interface problems, they can typically only estimate partial error components \cite{wu2023convergence,ying2024accurate}, and a comprehensive estimate of the total error is still lacking. This work aims to bridge this gap within the framework of randomized neural networks (RaNNs) \cite{gallicchio2020deep,pao1994learn,scardapane2017randomness}.

RaNNs form a class of feedforward neural networks in which the hidden-layer parameters are randomly initialized and kept fixed, so that only the output layer is trainable. This class of models unifies a range of techniques, including extreme learning machines \cite{trends2015huang,huang2006extreme,liu2015extreme}, random feature methods \cite{ChenChiEYang2022,weinan2020comparative,rahimi2007random}, and stochastic basis selection \cite{igelink1995stochastic}. Within RaNN framework, the nonconvex training of a general neural network can be reduced to a least-squares problem. This simplification is particularly important for interface problems, where loss terms arising from interface conditions can be highly imbalanced, thereby complicating the training landscape \cite{chi2024random,li2025local,ying2024accurate}. As the optimization reduces to a convex optimization problem, the primary analytical challenge shifts to approximation and generalization errors.

Approximations for RaNNs are typically obtained by assuming that the target function admits a Fourier-type representation of the form $\int_{\R^d} e^{i \vec{\omega} \cdot \x } G(\vec{\omega})d \vec{\omega}$. Then an activation-based integral representation is derived and subsequently approximated by random sampling \cite{gonon2023random, gonon2023approximation,xu2025priori,ying2024accurate}. To prove such integral representations, one line of research introduces the assumption that the coefficient function $G(\vec{\omega})$ satisfies certain integrability conditions determined by the hidden-parameter distribution, under which an $L^p$-approximation is derived \cite{gonon2023random, gonon2023approximation,ying2024accurate}. For interface problems, however, controlling PDE residuals and interface conditions requires approximating second-order derivatives, which in turn demands stronger and potentially restrictive assumptions on $G(\vec{\omega})$ and the hidden-parameter distributions. Another line of work modifies the commonly used activation function to ensure absolute integrability, thereby enabling Fourier-analytic arguments and yielding the corresponding integral representation \cite{siegel2020approximation,xu2025priori}. In practice, however, the $\tanh$ activation, which does not belong to $L^1(\R)$, is widely adopted for elliptic interface problems, owing to its simplicity and demonstrating empirical effectiveness \cite{chi2024random,li2025local,ying2024accurate}.

Generalization error is typically characterized by estimating the complexity of the underlying hypothesis class \cite{gyorfi2006distribution}. Recent studies have established such bounds for the Deep Ritz method \cite{e2018deep} and PINNs \cite{raissi2019physics} under various settings \cite{duan2022convergence,jiao2025drm,lu2022machine,lu2021priori,shin2020on}. Most of the results report a convergence rate of order $\mathcal{O}(n^{-1/2})$, where $n$ denotes the number of training samples \cite{duan2022convergence,jiao2025drm,lu2021priori}. Faster $\mathcal{O}(n^{-1})$ rates have also been obtained in a qualitative high-probability sense under the assumption that the boundary conditions are exactly satisfied \cite{lu2022machine}. In contrast, interface problems involve jump conditions that cannot be assumed to hold exactly and call for fully quantitative and tighter generalization bounds.

\textbf{Contributions.} The main contributions of this work are summarized as follows.

\begin{itemize}
\item We establish, to the best of our knowledge, the first $\tanh$-based integral representation for Barron functions valid under differentiation up to second order (see \cref{thm:integral_expression}). The main ingredients are the local smoothness of the Fourier transform of $\tanh(x)/(1 + x^2)$ (see \cref{lem:tanh}) and tailored Fourier-analytic arguments. This representation leads to a $C^2$-approximation bound for RaNNs (see \cref{thm:approximation_thm}) that requires no restrictive assumptions on the hidden-parameter distribution.
Moreover, the underlying framework extends to a broad class of non-integrable activation functions.

\item We establish a quantitative $\mathcal{O}(n^{-1})$ generalization bound for the RaNN-based loss function that jointly enforces the PDE, boundary, and interface conditions by extending localization techniques developed in \cite{bartlett2005local} (see \cref{thm:gen}). This result refines previous analyses that report $\mathcal{O}(n^{-1/2})$ rates or qualitative high-probability guarantees \cite{lu2022machine,lu2021priori,ying2024accurate}.

\item Combining the above analysis with the consistency and convex quadratic structure of the loss function, we establish the first comprehensive error estimate of neural network methods for solving elliptic interface problems with Barron-type solutions, accounting for approximation, generalization, and optimization errors (see \cref{thm:main}).

\end{itemize}

The paper is organized as follows. Section \ref{sec:preliminary} introduces the preliminaries. Section \ref{sec:main_result} presents the RaNN method and the error analysis, with proofs given in section \ref{sec:proof}. Numerical experiments are reported in section \ref{sec:numerical_results}, and section \ref{sec:conclusions} concludes the paper.

\section{Preliminaries}\label{sec:preliminary}

In this section, we introduce the mathematical notation, formulate the elliptic interface problem, and provide a brief review of neural network methods within a domain decomposition framework for solving such problems.

\subsection{Notations}
Throughout this paper, $\Omega\subset\R^d$ denotes a bounded domain. For $k\in\mathbb{N}$, define
\(
C^k(\Omega)=\left\{f: D^{\vec{k}}f\in C(\Omega) \text{ for all}\ |\vec{k}|\leq k \right\}
\)
and let $C_c^\infty(\Omega)$ denote the space of smooth functions with compact support in $\Omega$. For $m\geq 0$, $H^m(\Omega)$ denotes the standard Sobolev space of real-valued functions whose weak derivatives up to order $m$ belong to $L^2(\Omega)$, and its norm is denoted by $\norm{\cdot}_{m,\Omega}$. The dual space of $H_0^m(\Omega)$ is denoted by $H^{-m}(\Omega)$; see \cite{adams2003sobolev} for details.

For $\generalfunc\in L^1(\R^d)$, the Fourier transform and its inverse are defined by
\[
 \F_d[\generalfunc](\vec{\omega})= \int_{\R^d} \generalfunc(\x) e^{-i \vec{\omega} \cdot \x } d \x,
\quad
 \F_d^{-1}[\F_d[\generalfunc]](\x) = \frac{1}{(2\pi)^d} \int _{\R^d} \F_d[\generalfunc](\vec{\omega} ) e^{i \vec{\omega} \cdot \x} d \vec{\omega}.
\]
The Barron space of order $s \ge 1$ is defined by \cite{xu2020finite}
{\small
\begin{equation}
\label{eq:barron_space}
B^s(\Omega)   = \left\{\generalfunc \in L^2(\Omega):\ \|\generalfunc\|_{B^s(\Omega)} = \inf_{\generalfunc_e|_{\Omega} = \generalfunc} \int_{\R^d} (1 + \|\vec{\omega}\|_{1})^s \left| \F_d[\generalfunc_e](\vec{\omega})\right| d\vec{\omega}
< \infty \right\},
\end{equation}
}
where the infimum is taken over all extensions \(\generalfunc_e \in L^{1}(\R^d)\).
For a fixed dataset $\{\x_i\}_{i=1}^n\subset\Omega$ and any $p\in\mathbb{N}$, let $\|\generalfunc\|_{n,p} := \left(\frac{1}{n}\sum_{i=1}^n |\generalfunc(\x_i)|^p\right)^{1/p}$ denote the empirical $L^p$-norm. The restriction of a function $\generalfunc$ to a subdomain $D \subset \Omega$ is denoted by $\generalfunc|_{D}$. The notation $a \lesssim b$ means that $a \le C b$ for some generic constant $C > 0$ independent of the primary parameters of interest (e.g., the sample size $n$ or network width $M$).

\subsection{Elliptic interface problem}
In this work, we consider the following elliptic interface problem:
\begin{equation}
\label{eq:interface_problem}
\begin{aligned}
-\nabla\cdot\left(\beta(\x) \nabla \generalfunc(\x)\right) = & g(\x) ,\quad  &&\x\in \Omega,\\
 \llbracket \beta(\x)  \nabla \generalfunc(\x) \cdot \mathbf{n} \rrbracket   = g_n(\x), \quad
\llbracket \generalfunc(\x)  \rrbracket  = & g_d(\x),\quad    &&\x\in \Gamma ,\\
\generalfunc(\x)= & g_b(\x),\quad  &&\x\in  \partial\Omega.
\end{aligned}
\end{equation}
Here, $\Gamma$ is a closed $C^2$-smooth interface that separates $\Omega$ into an interior subdomain $\Omega_1$ and an exterior subdomain $\Omega_2=\Omega\setminus\Omega_1$. Throughout the paper, for any function $v: \Omega \rightarrow \R$, we denote its restriction to $\Omega_i$ by $v_i$ for $i = 1, 2$. The coefficient $\beta(\x)>0$ is assumed to be piecewise constant with high contrast, i.e., either \( \beta_1/\beta_2  \gg 1\) or \( \beta_2/\beta_1  \gg 1\). The jump of $\generalfunc$ across $\Gamma$ is defined as $\llbracket \generalfunc \rrbracket := \generalfunc_2 - \generalfunc_1$, and $\mathbf{n}$ denotes the unit normal on $\Gamma$ pointing toward $\Omega_2$. Assuming that $g\in L^2(\Omega)$, $g_n\in H^{1/2}(\Gamma)$, $g_d\in H^{3/2}(\Gamma)$, and $g_b\in H^{3/2}(\partial\Omega)$, the interface problem admits a unique solution $\solufunc\in X^{2}:=H^2(\Omega_1)\cap H^2(\Omega_2)\cap L^\infty(\Omega)$, equipped with the norm \(\|\cdot\|_{X^{2}(\Omega)}^2 = \|\cdot\|_{H^2(\Omega_1)}^2 + \|\cdot\|_{H^2(\Omega_2)}^2\), see \cite{chen1998finite,huang2002some, wu2023convergence}. In addition, to ensure that the loss function introduced later is pointwise well-defined, we further assume $g_i\in C^0(\Omega_i)$, $g_n\in C^0(\Gamma)$, $g_d\in C^0(\Gamma)$, and $g_b\in C^0(\partial\Omega)$.

\subsection{Neural network method for interface problems\label{subsec:nn_method_for_inp}}
The interface problem~\eqref{eq:interface_problem} typically exhibits distinct solution behaviors in the two subdomains, motivating the use of neural network methods that combine domain decomposition with physics-informed learning \cite{chi2024random,he2022mesh,li2025local,wu2022interfaced}. In these approaches, the neural network solution $\generalNN(\x;\vec{\theta})$ is represented by two distinct subnetworks on the subdomains $\Omega_1$ and $\Omega_2$, denoted by $\generalNN_1(\x;\vec{\theta}_1)$ and $\generalNN_2(\x;\vec{\theta}_2)$, respectively.
Here, $\vec{\theta}=\{\vec{\theta}_1,\vec{\theta}_2\}$ denotes the set of all trainable parameters. Let
\begin{equation*}
\begin{aligned}
  &\ell_{i}(\x)\big|_{\Omega_i}= ( \beta_i \Delta \generalNN_i(\x) + g_i(\x))^2,\ i=1,2,  &&\ell_{3}(\x)\big|_{\Gamma} =\left( \llbracket \beta\nabla \generalNN(\x) \cdot \mathbf{n} \rrbracket - g_n(\x)\right)^2,  \\
 &\ell_{4}(\x)\big|_{\Gamma} = \left(\llbracket \generalNN(\x) \rrbracket - g_d(\x)\right)^2,
 &&\ell_{5}(\x) \big|_{\partial\Omega}=\left(\generalNN_2(\x)-g_b(\x)\right)^2,
 \label{eq: definition of terms in localized space}
\end{aligned}
\end{equation*}
be the loss terms enforcing the PDE in each subdomain, the interface jump conditions, and the boundary condition. Solving \cref{eq:interface_problem} is then reformulated as minimizing the expected loss function:
\begin{equation}
\label{eq:expected_loss_function}
 L(\generalNN; \vec{\theta})=\sum_{i=1}^5 \lambda_i^2 \mathbb{E}[\ell_i(\x)],
\end{equation}
where the expectation $\mathbb{E}[\cdot]$ is taken with respect to the uniform distribution over the corresponding domain, and the constants $\lambda_i$ ($i=1,\ldots,5$) are penalty weights that balance the contributions of the loss terms. In practice, the expectations are approximated by empirical averages using the following i.i.d. sample sets:
\begin{equation}
\label{eq:training_datasets}
\begin{aligned}
\left\{\x_{1,j}\right\}_{j=1}^{n_1} &\overset{\text{i.i.d.}}{\sim} \text{Unif}(\Omega_1),\
\left\{\x_{2,j}\right\}_{j=1}^{n_2} \overset{\text{i.i.d.}}{\sim} \text{Unif}(\Omega_2),\
\left\{\x_{3,j}\right\}_{j=1}^{n_3} \overset{\text{i.i.d.}}{\sim} \text{Unif}(\Gamma),\\
&\left\{\x_{4,j}\right\}_{j=1}^{n_4} \overset{\text{i.i.d.}}{\sim} \text{Unif}(\Gamma),\
\left\{\x_{5,j}\right\}_{j=1}^{n_5} \overset{\text{i.i.d.}}{\sim}
\text{Unif}(\partial\Omega).
\end{aligned}
\end{equation}
We assume throughout that the sample sizes are comparable, namely, $n_i = \mathcal{O}(n)$. Writing \(\mathbf{n}=(n_1,\ldots,n_5)\), the resulting empirical loss is
\begin{equation}
L^{\mathbf{n}}(\generalNN; \vec{\theta})  = \sum_{i=1}^5 \frac{\lambda_i^2}{n_i}\sum_{j=1}^{n_i} \ell_{i}(\x_{i,j}).
\label{eq:empirical_loss_function}
\end{equation}
For general neural networks, $L^{\mathbf{n}}(\generalNN;\vec{\theta})$ is a nonlinear and nonconvex function of $\vec{\theta}$; hence, it is typically optimized via gradient-based methods.

\section{Main results\label{sec:main_result}}
In this section, we first establish the consistency of the loss function \eqref{eq:expected_loss_function}. We then present a rigorous error analysis of the resulting randomized neural network (RaNN) estimator by decomposing its excess risk into approximation and generalization components.

\subsection{RaNN method for elliptic interface problems}
Given an activation function $\sigma$ and parameters $M,q,C_B>0$, let $\pi_{\vec{\omega},b}$ be any strictly positive density supported on $[-q,q]^{d+1}$. We consider the following shallow RaNN function class:
\begin{equation}
\label{def:stable_shallow_rnn_space}
\Sigma_{M,q,C_B}^\sigma (\Omega) = \left\{
\sum_{m=1}^M a_m\sigma(\vec{\omega}_{m} \cdot\x+b_{m}) \;\middle|\;
   \begin{array}{l}
     \x \in \Omega, \\
     (\vec{\omega}_{m},b_m)\sim \pi_{\vec{\omega},b}, \\
     \sum_{m=1}^M (1+ \|\vec{\omega}_{m}\|)^2|a_m|\leq C_B
   \end{array}
\right\}.
\end{equation}
Unlike standard fully trained networks, the hidden-layer parameters $\{(\vec{\omega}_{m},b_m)\}_{m=1}^M$ are fixed upon random initialization; only the output-layer $\{a_m\}_{m=1}^M$ are trainable.

We now present the RaNN method for elliptic interface problems within the framework described in \cref{subsec:nn_method_for_inp}. For simplicity, RaNNs defined on $\Omega_1$ and $\Omega_2$ are assumed to share the same set of parameters $(M, q, C_B)$. With $\sigma$ chosen to be $\tanh$, the randomized hypothesis space is defined as
\begin{equation}
\label{eq:hypothesis_space_rnn}
F(\Omega) = \Sigma_{M,q,C_B}^{\tanh}  (\Omega_1)  \cap \Sigma_{M, q, C_B}^{\tanh} (\Omega_2),
\end{equation}
and each $\generalNN(\x) \in F(\Omega)$ admits the representation:
\begin{equation*}
\generalNN(\x)=
\begin{cases}
 \sum_{m=1}^{M} a_m^1 \Phi^1_m(\x): = \sum_{m=1}^{M} a_m^1 \tanh(\vec{\omega}_{m}^1 \cdot\x + b_{m}^1),  & \x\in \Omega_1, \\
 \sum_{m=1}^{M} a_m^2 \Phi^2_m(\x): = \sum_{m=1}^{M} a_m^2 \tanh(\vec{\omega}_{m}^2\cdot \x + b_{m}^2),  & \x\in \Omega_2.
\end{cases}
\end{equation*}
Since the sample sizes $n_i$ in \eqref{eq:empirical_loss_function} are chosen to be of comparable sizes, the normalization factors $1/n_i$ can be omitted in the implementation. Moreover, because $F(\Omega)$ depends linearly on the coefficient vector $\vec{a}:=\left[a^1_1,\ldots,a^1_M,a^2_1,\ldots,a^2_M\right]^T$, minimizing the empirical loss \eqref{eq:empirical_loss_function} reduces to a linear least-squares problem:
\begin{equation}
\label{eq:empirical_loss_function_of_rann}
   \hat{\vec{a}} = \min_{\vec{a}}\|A\vec{a}-\vec{b}\|_2^2, \quad  \text{s.t.} \quad \sum_{m=1}^{M} (1 + \|\vec{\omega}_m^i\|)^2 |a_m^i| \le C_B,
    \quad i = 1,2.
\end{equation}
Here
\(\vec{b}=\left[\lambda_1 g_1(\x_{1,j})_{j=1}^{n_1}, \lambda_2 g_2(\x_{2,j})_{j=1}^{n_2},
\lambda_3 g_n(\x_{3,j})_{j=1}^{n_3},
\lambda_4 g_d(\x_{4,j})_{j=1}^{n_4},
\lambda_5 g_b(\x_{5,j})_{j=1}^{n_5}\right]^T,
\)
and the matrix $A$ has the block structure:
\begin{equation*}
 \mathbf{A} =
 \left[
 \begin{array}{cc}
 \lambda_1 \left[-\nabla\cdot(\beta_1\nabla \Phi_m^1(\x_{1,j}))\right]_{n_1\times M}, &\left[\mathbf{0}\right]_{n_1\times M} \\
 \left[\mathbf{0}\right]_{n_2\times M}, &\lambda_2 \left[-\nabla\cdot(\beta_2\nabla \Phi_m^2(\x_{2,j}))\right]_{n_2\times M} \\
 -\lambda_3 \left[\beta_1\frac{\partial\Phi_m^1(\x_{3,j})}{\partial \mathbf{n}}\right]_{n_3\times M}, &\lambda_3 \left[\beta_2\frac{\partial\Phi_m^2(\x_{3,j})}{\partial \mathbf{n}}\right]_{n_3\times M} \\
 -\lambda_4 \left[\Phi_m^1(\x_{4,j})\right]_{n_4\times M}, & \lambda_4 \left[\Phi_m^2(\x_{4,j})\right]_{n_4\times M}\\
 \left[\mathbf{0}\right]_{n_5\times M}, & \lambda_5 \left[\Phi_m^2(\x_{5,j})\right]_{n_5\times M}
 \end{array}
\right].
\end{equation*}

Let $\hat{\generalNN}\in F(\Omega)$ denote the RaNN estimator corresponding to the solution $\hat{\vec{a}}$ of \eqref{eq:empirical_loss_function_of_rann}. We next establish the consistency of the RaNN method, showing that $\hat{\generalNN}$ converges to the exact solution $\generalfunc^*$ as the loss $L(\hat{\generalNN})$ approaches zero.

\begin{theorem}\label{thm:consistency}
The exact solution $\generalfunc^*$ of the interface problem \eqref{eq:interface_problem} satisfies an a priori bound
\begin{equation}
\label{eq:priori_estimate_rnn}
\begin{aligned}
 \|\generalfunc^* \|_{0,\Omega}   \lesssim & \left(\tfrac{1}{\beta_1} + T_\beta\right)\|g\|_{0,\Omega_1}
 + \left(\tfrac{1}{\beta_2} + T_\beta\right)\|g\|_{0,\Omega_2}
 + T_\beta\|g_n\|_{0 ,\Gamma}  \\ &+ \left(1+T_\beta \min\{\beta_1,\beta_2\} \right)\|g_d\|_{0,\Gamma}
 + \left(1+T_\beta \beta_2\right)\|g_b\|_{0,\partial\Omega},
\end{aligned}
\end{equation}
where $T_\beta=1/\max\{\beta_1,\beta_2\}$.
Moreover, for any $\generalNN\in F(\Omega)$,
\begin{equation}
\label{eq:the_first_ieq}
    \|\generalNN-\generalfunc^*\|_{0,\Omega}^2   \;\lesssim\; L(\generalNN).
\end{equation}
\end{theorem}
\begin{proof}
The estimate \eqref{eq:priori_estimate_rnn} is obtained by applying the estimate in \cite{huang2002some}, and adapting the argument in \cite{wu2023convergence}, with explicit dependence on the coefficient $\beta$; a detailed derivation is provided in Appendix \ref{proof:prior}.

To prove \eqref{eq:the_first_ieq}, we apply \eqref{eq:priori_estimate_rnn} to the error function $\generalNN-\generalfunc^*$:
    \begin{align*}
        \|\generalNN -\generalfunc^* \|_{0,\Omega}  \lesssim & \|\beta_1 \Delta \generalNN _1 + g_1\|_{0,\Omega_1} + \left\|\beta_2 \Delta \generalNN _2 + g_2\right\|_{0,\Omega_2} + \left\|\llbracket \beta \nabla \generalNN  \cdot \mathbf{n} \rrbracket - g_n\right\|_{0,\Gamma} \\ & + \left\|\llbracket \generalNN  \rrbracket - g_d\right\|_{0,\Gamma} + \|\generalNN _2 - g_b\|_{0,\partial\Omega}.
    \end{align*}
By the definition of $L(\generalNN)$, squaring both sides and applying the AM-QM inequality give the desired estimate.
\end{proof}

\begin{remark}
\label{remk:penalty_weights}
In addition to the consistency analysis, estimate~\cref{eq:priori_estimate_rnn} suggests a natural scaling of the penalty weights with respect to the coefficient contrast. Indeed, the coefficients multiplying the data \(g_d\) and \(g_b\) satisfy
\[
1+T_\beta\min\{\beta_1,\beta_2\}\leq 2,
\quad
1+T_\beta\beta_2\leq 2.
\]
Thus, contrast-dependent scaling is retained only for the first three loss components, as indicated by the a priori estimate, while coefficients of the last two terms can be chosen as fixed contrast-independent constants. Accordingly, we choose \begin{equation}\label{eq:penalty_weights}
\lambda_1=\frac{1}{\beta_1}+T_\beta,\quad
\lambda_2=\frac{1}{\beta_2}+T_\beta,\quad
\lambda_3=T_\beta,\quad
\lambda_4=\lambda_5=10.
\end{equation}
This choice yields robust numerical performance over the range of coefficient contrasts considered in~\cref{sec:numerical_results}.
\end{remark}

\subsection{Theoretical analysis}
We now state the main theorem, which establishes an end-to-end error analysis for solving elliptic interface problems via RaNNs. The complete proof is given in \cref{sec:proof}.

\begin{theorem}
\label{thm:main}
Let $M,d,n\in \mathbb{N}^+$ and $q>0$, and let $\solufunc:\Omega\rightarrow \R$ denote the exact solution to the interface problem~\eqref{eq:interface_problem}. Assume that $\generalfunc_i^*\in B^4(\Omega_i)$ for $i=1,2$. Consider the randomized hypothesis space $F(\Omega)$ in \eqref{eq:hypothesis_space_rnn} with $C_B\geq C_q$, and let $\hat \generalNN$ be the RaNN estimator obtained by solving the linearly constrained least-squares problem~\eqref{eq:empirical_loss_function_of_rann} on the datasets \eqref{eq:training_datasets} with sample sizes satisfying $n_i = \mathcal{O}(n)$. Then,
\begin{align*}
\mathbb{E}\left[\|\hat{\generalNN} - \solufunc\|_{0, \Omega}^2\right] \lesssim \frac{1}{q^2} + \frac{I_q}{M} + \frac{ C_B^4 M \log\left( C_BM(q+1) n \right) }{n}.
\end{align*}
Here, $\mathbb{E}[\cdot]$ denotes the expectation with respect to the random sampling of the datasets and the hidden-layer parameters, and \(C_q, I_q < \infty \) are constants defined by the density \(\pi_{\vec{\omega},b}\) supported on \([-q,q]^{d+1}\).
\end{theorem}

Owing to the convexity of the underlying optimization problem, the main task is to derive upper bounds for the approximation and generalization errors, which correspond, respectively, to the first two terms and the last term in the error bound of \cref{thm:main}. A key difference between our analysis and that of standard feedforward neural networks lies in the treatment of the approximation error.

\begin{theorem}\label{thm:approximation_thm}
Let $q>0$, $M, d\in \mathbb{N}^+$. Consider the randomized hypothesis space $\Sigma_{M,q,C_B}^{\tanh}(\Omega)$ defined by \eqref{def:stable_shallow_rnn_space} with $C_B\geq C_q$.
Then, for any \(\generalfunc \in B^4(\Omega)\),
it holds that
\[
\mathbb{E}_{\{\vec{\omega}, b\}} \!\left[
\inf_{\generalNN \in \Sigma_{M,q,C_B}^{\tanh}(\Omega)} \left\| \generalNN - \generalfunc \right\|^2_{C^2(\Omega)}
\right]
\lesssim \frac{1}{q^2} + \frac{I_q}{M},
\]
where $\mathbb{E}_{\{\vec{\omega}, b\}}[\cdot]$ denotes the expectation with respect to the random sampling of the hidden-layer parameters, and $C_q, I_q<\infty$ are constants defined by the density \(\pi_{\vec{\omega},b}\) supported on \([-q,q]^{d+1}\). In particular, there exists a choice of parameter distribution for which \( C_q, I_q \) are independent of \(q\).
\end{theorem}

As shown in \cref{thm:approximation_thm}, the approximation error of RaNNs depends not only on the network width $M$, but also on the distribution of the hidden parameters, as reflected by the constants  $q$ and $I_q$.
The proof relies on the $\tanh$-based integral representation introduced in \cref{thm:integral_expression}. A similar integral-representation viewpoint also appears in several approximation results for standard fully trained shallow neural networks, where random sampling is employed merely as an auxiliary step to prove the existence of suitable network weights \cite{barron1993universal, caragea2023neural, ma2019priori, siegel2020approximation}.
In these works, the approximation rate over Barron-type spaces is typically of order $\mathcal{O}(1/M)$, while the hidden-parameter distributions often depend implicitly on the unknown target function. In contrast, \cref{thm:approximation_thm} accommodates any hidden-parameter distribution supported on a cube with strictly positive density, thereby facilitating a priori, target-independent sampling procedures. In addition, it applies directly to the standard $\tanh$ activation function, which is inherently compatible with practical implementations \cite{chi2024random, li2025local, ying2024accurate}.

We next describe the generalization results of RaNNs for solving interface problems, thereby yielding the last error component in \cref{thm:main}.

\begin{theorem}\label{thm:gen}
Let $\hat{\generalNN}$ and $\solufunc$ denote the RaNN estimator and the exact solution, respectively, as in \cref{thm:main}. The following estimate holds:
\[
\mathbb{E}_{\{\x\}}\big[L(\hat \generalNN)\big] \ \lesssim\
 \inf_{\generalNN \in F(\Omega)}   \|\generalNN - \generalfunc^* \|^2_{X^{\infty}( \Omega)} +  \frac{C_B^4 M \log (C_BM(q+1)n)}{n},
\]
where $\mathbb{E}_{\{\x\}}[\cdot]$ denotes the expectation with respect to the random sampling of datasets, and the $C^2$-based composite norm is defined by
$ \|\cdot \|^2_{X^{\infty}(\Omega)} =  \| \cdot\|^2_{C^2(\overline{\Omega}_1)} + \|\cdot \|^2_{C^2(\overline{\Omega}_2)}$.
\end{theorem}

\Cref{thm:gen} establishes a generalization bound of order $\mathcal{O}(n^{-1})$, improving upon the $\mathcal{O}(n^{-1/2})$ rates typically reported in prior work \cite{duan2022convergence,jiao2025drm,lu2021priori, ying2024accurate}. The main ingredient of the proof lies in extending the localization techniques developed in \cite{bartlett2005local}, which refine complexity control by restricting the analysis to a data-dependent localized hypothesis class. Similar localization arguments have been used in \cite{lu2022machine} for elliptic PDEs.
In contrast, our work makes distinct contributions by addressing interface problems and by providing a detailed probabilistic characterization of the generalization error (see \cref{lem:gen}); the latter yields an explicit bound in expectation that enables quantitative error evaluation.

\section{Proof of main results\label{sec:proof}}
\Cref{thm:main} is a direct consequence of Theorems~\ref{thm:consistency}, \ref{thm:approximation_thm} and \ref{thm:gen}.
\begin{proof}[Proof of \cref{thm:main}]
\begin{equation*}
\begin{aligned}
\mathbb{E}\left[\|\hat{\generalNN} - \generalfunc^*\|_{0, \Omega}^2\right]
\underset{\text{(\cref{thm:consistency})}}{\lesssim} &\ \mathbb{E}\left[L(\hat{\generalNN})\right] \\
\underset{\text{(\cref{thm:gen})}}{\lesssim}&\ \mathbb{E}_{\{\vec{\omega}, b\}} \left[ \inf_{u \in F(\Omega)} \|\generalNN - \generalfunc^* \|^2_{X^{\infty}} \right]\! +\! \frac{C_B^4 M \log\left( C_BM(q+1) n \right) }{n}\\
\underset{\text{(\cref{thm:approximation_thm})}}{\lesssim} &\ \frac{1}{q^2} + \frac{I_q}{M} + \frac{C_B^4 M \log\left( C_BM(q+1) n  \right) }{n},
\end{aligned}
\end{equation*}
where $\mathbb{E}[\cdot]$ denotes the expectation with respect to both the random sampling of datasets and the hidden-layer parameters.
\end{proof}
The remainder of this section is devoted to the proofs of Theorems~\ref{thm:approximation_thm} and \ref{thm:gen}, which quantify the approximation and generalization errors of RaNNs for elliptic interface problems in subsections \ref{sec:appro} and \ref{sec:genera}, respectively.

\subsection{Approximation error}\label{sec:appro}

We begin by introducing the assumptions on the activation function used in the analysis.

\begin{assumption}\label{ass:activation1}
The activation function $\sigma \in C^2(\R)$ is non-constant, and $\sigma$ and its first and second derivatives are bounded and Lipschitz continuous on $\R$.
\end{assumption}

\begin{assumption}\label{ass:activation2}
The activation function \( \sigma \) admits the factorization
\(
\sigma(x) = (1 + x^2)^t \psi(x) = \langle x \rangle^t \psi(x)
\)
for a constant \( t \geq 0 \), where \( \psi, \mathcal{F}_1[\psi] \in L^1(\R)\); \( \mathcal{F}_1[\psi] \) is smooth and not identically zero on a closed interval \( [w_1, w_2] \) that excludes the origin.
\end{assumption}

Note that the above assumptions are mild; \cref{lem:tanh} confirms that the commonly used activation \( \tanh(x) \) satisfies all the required conditions.

\begin{lemma}
\label{lem:tanh}
The activation function \( \tanh(x) \) satisfies Assumptions~\ref{ass:activation1} and~\ref{ass:activation2}.
\begin{proof}
By definition, \( \tanh(x)\in  C^\infty(\R) \) satisfies all the conditions in \cref{ass:activation1}. We next verify \cref{ass:activation2}. Define \(\psi(x) :=  \tanh(x)/(1 + x^2)\), which corresponds to the case \( t = 1 \). Since $\tanh(x)$ is bounded and $(1 + x^2)^{-1}$ decays quadratically, it follows that \( \psi \) and its first and second derivatives all belong to \( L^1(\R)\). Standard properties of the Fourier transform then imply $|\mathcal{F}_1[\psi](w)| \leq \|\psi\|_{L^1(\R)}$ and $|\mathcal{F}_1[\partial_x^2 \psi](w)| \leq \|\partial_x^2 \psi\|_{L^1(\R)}.$
Using the identity
\(
w^2 \mathcal{F}_1[\psi](w) = -\mathcal{F}_1[\partial_x^2 \psi](w),
\)
we obtain
\[
(1 + w^2)\, |\mathcal{F}_1[\psi](w)| \leq \|\psi\|_{L^1(\R)} + \|\partial_x^2 \psi\|_{L^1(\R)} := C,
\]
which yields the decay estimate $|\mathcal{F}_1[\psi](w)| \leq \frac{C}{1 + w^2}.$ Hence, \( \mathcal{F}_1[\psi](w) \in L^1(\R) \). In addition, \( \psi(x) \) is odd, and its Fourier transform admits the representation
\[
\begin{aligned}
   \mathcal{F}_1[\psi](w)
   &= -2i \int_0^{\infty} \frac{\sin(xw)}{1 + x^2} dx + 4i \int_0^{\infty} \frac{\sin(xw)}{(1 + x^2)(e^{2x} + 1)} dx.
\end{aligned}
\]
For the first integral, an explicit formula is available \cite{gradshteyn2014table}:
\[
\int_0^{\infty} \frac{\sin(w x)}{1 + x^2} dx =
\begin{cases}
\frac{1}{2} \left[e^{-w} \overline{\mathrm{Ei}}(w) - e^{w} \mathrm{Ei}(-w)\right], & w > 0, \\
0, & w = 0, \\
-\frac{1}{2} \left[e^{w} \overline{\mathrm{Ei}}(-w) - e^{-w} \mathrm{Ei}(w)\right], & w < 0,
\end{cases}
\]
where \( \mathrm{Ei}(w) \) denotes the exponential integral \cite{gradshteyn2014table}.
The second integral decays exponentially in \( x \), ensuring absolute convergence for all \( w \). Thus, \( \mathcal{F}_1[\psi](w) \) is well-defined and integrable on any compact interval excluding \( w = 0 \).
\end{proof}
\end{lemma}

\begin{remark}
By arguments analogous to those used for $\tanh(x)$, one can verify that other widely used activation functions, such as the sigmoid $\sigma(x) = (1+e^{-x})^{-1}$ and the sine $\sigma(x) = \sin(x)$, also satisfy Assumptions~\ref{ass:activation1} and~\ref{ass:activation2}.
\end{remark}

The main ingredient for approximation is the following integral representation. As described in \cref{subsec:nn_method_for_inp}, the approximations over $\Omega_1$ and $\Omega_2$ can be treated independently; therefore, it suffices to analyze RaNNs on a single subdomain. For notational convenience, we use $\Omega$ to denote a generic bounded subdomain.

\begin{theorem}
\label{thm:integral_expression}
Suppose that the activation function $\sigma$ satisfies Assumptions~\ref{ass:activation1} and~\ref{ass:activation2}. For any multi-index $\vec{k} \in \mathbb{N}^d$ with $|\vec{k}| \leq 2$ and any \(\generalfunc \in B^s(\Omega)\) with $|\vec{k}|+1 \leq s$, there exists a function
\(\coeff : \R^d \times \R \to \R\) such that
\begin{equation*}
D^{\vec{k}} \generalfunc(\x ) = \int_{\R^d} \int_{\R} D^{\vec{k}} \sigma(\vec{\omega} \cdot \x  + b)\, \coeff(\vec{\omega}, -b)\, db\, d\vec{\omega},
\quad \forall \x  \in \Omega.
\end{equation*}
\end{theorem}
\begin{proof}
Since the activation function $\sigma$ satisfies \cref{ass:activation2}, we can construct a smooth and compactly supported auxiliary function $\rho(w)$ such that
\[
C^{-1} := (2\pi)^{d-1} \int_{\R} \langle i\partial_{w}\rangle^t \left[\rho(w)\right] \, \F_1[\psi](w) \, dw
\]
is finite and nonzero, where $\langle i\partial_{w}\rangle^t := (1 - \partial_w^2)^t$ denotes a linear differential operator. For example, let $\epsilon > 0$ be arbitrarily small. There exists a smooth function $\rho^a$ that equals the complex conjugate of $\F_1[\psi](w)$ on the closed interval $[w_1+\epsilon, w_2-\epsilon]$ and vanishes outside the interval $[w_1, w_2]$ (as defined in \cref{ass:activation2}) . If $t = 0$, we set $\rho(w) := \rho^a(w)$; otherwise, let $\rho^b$ be a smooth solution of $\langle i\partial_{w}\rangle^t \rho^b = \rho^a$. Finally, define $\rho(w) := \chi(w) \rho^b(w)$, where $\chi$ is a bump function supported in $[w_1, w_2]$ and identically one on $[w_1+\epsilon, w_2-\epsilon]$. By construction, $\rho\in C^{\infty}_c(\R)$ satisfies the requirements.

For \(\generalfunc \in B^s(\Omega)\), choose $\generalfunc_A\in L^1(\R^d)$ such that $\norm{\generalfunc_A}_{B^s(\Omega)} \leq 2\norm{\generalfunc}_{B^s(\Omega)}$ and $\chi_A\in C_c^\infty(\R^d),\ \chi_A\equiv1$ on $\overline\Omega$. Then we define
\[
\generalfunc_E: = \chi_A\generalfunc_A,\quad \coeff^{\#}(\vec{\omega} ,w) := C \F_d [\generalfunc_E] \left(\vec{\omega} w\right) \rho(w) |w|^d,\]
which yields $\generalfunc_E|_{\Omega} = \generalfunc$ and $\ \norm{\generalfunc_E}_{B^s(\Omega)} \lesssim \norm{\generalfunc}_{B^s(\Omega)}$. Subsequently, we obtain that
\[
\begin{aligned}
    \F_d [\generalfunc_E](\vec{\omega} ) &= C (2\pi)^{d-1} \int_{\R}\F_d [\generalfunc_E](\vec{\omega} ) \ \langle i\partial_{w}\rangle^t \left[ \rho(w) \right] \ \F_1 [\psi](w) dw \\
    & = (2\pi)^{d-1}\int_{\R} \langle i\partial_{w}\rangle^t \left[\coeff^{\#}(\vec{\omega} /w, w) |w|^{-d}\right] \F_1 [\psi](w) dw.
\end{aligned}
\]
Since $\F_d[\generalfunc_E] \in L^1(\R^d)$ by the definition of the Barron norm \eqref{eq:barron_space}, its inverse Fourier transform is well-defined. Therefore, $\generalfunc$ admits the representation
\begin{align*}
\generalfunc(\x) =&(2\pi)^{-1} \int_{\R^d} \int_{\R} \langle i\partial_{w}\rangle^t \left[\coeff^{\#}(\vec{\omega} /w, w) |w|^{-d}\right] \F_1 [\psi](w) dw e^{i \vec{\omega} \cdot \x} d \vec{\omega}, \  \forall \x\in \Omega.
\end{align*}
By Fubini’s theorem and a change of variables, we can rewrite $\generalfunc(\x)$ as follows:
\begin{align*}
\generalfunc(\x)
=& (2\pi)^{-1} \int_{\R } \langle i\partial_{w}\rangle^t \left[\int_{\R^d} \coeff^{\#}(\vec{\omega} /w , w ) |w |^{-d} e^{i \vec{\omega} \cdot \x} d \vec{\omega} \right] \F_1 [\psi](w ) dw \\
=& (2\pi)^{-1} \int_{\R } \langle i\partial_{w}\rangle^t \left[\int_{\R^d}\coeff^{\#}(\vec{\omega} , w ) e^{iw  \vec{\omega} \cdot \x} d \vec{\omega} \right] \F_1 [\psi](w ) dw .
\end{align*}
Defining
\begin{equation}
\begin{aligned}\label{eq:gamma}
\coeff(\vec{\omega} , b) = \F_1 ^{-1}\left[\coeff^{\#}\right](\vec{\omega} , b), \quad
\coeff^{\#}(\vec{\omega} , w ) = \F_1 [\coeff](\vec{\omega} , w ) = \int _{\R} \coeff(\vec{\omega} , b) e^{ - ib w } d b,
\end{aligned}
\end{equation}
we obtain
\begin{equation*}
\begin{aligned}
\generalfunc(\x)
=& (2\pi)^{-1} \int_{\R } \langle i\partial_{w}\rangle^t \left[\int_{\R^d}\int _{\R} \coeff(\vec{\omega} , b) e^{ - ib w } d b e^{iw  \vec{\omega} \cdot \x} d \vec{\omega} \right] \F_1 [\psi](w ) dw \\
=& (2\pi)^{-1} \int_{\R } \int_{\R^d}\int _{\R} \langle i\partial_{w}\rangle^t \left[ \coeff(\vec{\omega} , b) e^{iw  (\vec{\omega} \cdot \x - b)}\right] d b d \vec{\omega} \F_1 [\psi](w ) dw \\
=& (2\pi)^{-1}
 \int_{\R^d}\int_{\R} \int_{\R} \coeff (\vec{\omega} , b) \langle \vec{\omega} \cdot \x - b\rangle^t e^{iw  (\vec{\omega} \cdot \x- b)} \F_1 [\psi](w ) dw  db d \vec{\omega} \\
=&
 \int_{\R^d} \int_{\R} \coeff (\vec{\omega} , -b)\langle \vec{\omega} \cdot \x +b\rangle^t \psi (\vec{\omega} \cdot \x+b) db d\vec{\omega}.
\end{aligned}
\end{equation*}
It remains to verify the absolute integrability of $\coeff^{\#}$ and $\coeff$, so that \eqref{eq:gamma} is well-defined. For a fixed $\vec{\omega}$, since $\rho(w)\in C_c^\infty(\R)$, we have:
\begin{align*}
\int_{ \R }\left|\coeff^{\#}(\vec{\omega} , w )\right| d w & = \int_{ \R }  \left|C\F_d [\generalfunc_E] \left(\vec{\omega} w \right) \rho(w ) |w |^d\right| d w
 \leq |C| \norm{\generalfunc_E}_{L_{1}}\int_{ \R } \left|\rho(w ) |w |^d\right| d w ,
\end{align*}
which implies that $\coeff^{\#} \in L^1(\R)$ with respect to $w$ and $\coeff$ is well-defined. To establish the absolute integrability of $\alpha$, it suffices to show that
\begin{equation}\label{eq:estimate_Wgammma}
\begin{aligned}
&\int_{\R^d} \int_{\R} |b|^{\gamma_1}\|\vec{\omega}\|^{\gamma_2}|\coeff(\vec{\omega} , b)| d b d \vec{\omega} \lesssim  \norm{\generalfunc}_{B^{s} (\Omega)},
\end{aligned}
\end{equation}
where $\gamma_1, \gamma_2\in \mathbb{N}$ and $1+\gamma_1+\gamma_2\leq s$. By the definition of $\alpha$, we have
\begin{equation}\label{eq:alpha}
\begin{aligned}
\coeff(\vec{\omega}, b)
=& \frac{C}{2\pi}\int_{\R} \int_{\R^d} \generalfunc_E(\x) e^{-i w \vec{\omega} \cdot \x} d \x\rho(w ) |w |^d e^{iw  b} d w \\
= &\frac{C}{2\pi} \int_{\R^d} \generalfunc_E(\x)  \F_1[\rho(w ) |w |^d](\vec{\omega} \cdot \x -b) d \x.
\end{aligned}
\end{equation}
We can assume that $\x \in\operatorname{supp}\generalfunc_E$ yields $|\x |\le R$. By \cref{eq:alpha} and the rapid decay of $\F_1[\rho(w ) |w |^d]$, for $h=d+\gamma_1+\gamma_2+2$, there exist $C_h$ and $c_{\rho}$ such that $|\alpha(\vec\omega,b)|\le C_h\,\|\generalfunc_E\|_{L^1}\,(1+|b|)^{-h}
$ for $|b|>c_{\rho}R\|\vec\omega\|:=T_b$.
Then we have
\begin{equation*}
\int_{|b|> T_b}|b|^{\gamma_1}|\alpha(\vec\omega,b)|db \lesssim \|\generalfunc_E\|_{L^1}\,(1+ T_b)^{\gamma_1+1-h},
\end{equation*}
which yields the boundedness of the far-region contribution since
\begin{equation*}
\begin{aligned}
&\int_{\R^d} \int_{|b|> T_b} |b|^{\gamma_1}\|\vec{\omega}\|^{\gamma_2}|\coeff(\vec{\omega} , b)| d b d \vec{\omega}\lesssim \|\generalfunc_E\|_{L^1} \int_{\R^d} (1+ T_b)^{\gamma_1+1-h}\|\vec\omega\|^{\gamma_2} d \vec{\omega}.
\end{aligned}
\end{equation*}
On the other hand, we have the following estimation:
\begin{equation*}
\begin{aligned}
&\int_{\R^d} \int_{|b|\le T_b} |b|^{\gamma_1}\|\vec{\omega}\|^{\gamma_2}|\coeff(\vec{\omega} , b)| d b d \vec{\omega} \\
\lesssim &\int_{\R^d} \int_{\R} \int_{|b|\le T_b} |b|^{\gamma_1}\|\vec{\omega}\|^{\gamma_2} |\F_d [\generalfunc_E](\vec{\omega}
w ) \rho(w ) |w |^d| d b d w   d \vec{\omega} \\
\lesssim &\int_{\R^d} \int_{\R}  \|\vec{\omega}\|^{1+\gamma_1+\gamma_2} |\F_d [\generalfunc_E](\vec{\omega}
) \rho(w ) |w |^{-(1+\gamma_1+\gamma_2)}| d w   d \vec{\omega}
\lesssim  \norm{\generalfunc}_{B^{s} (\Omega)}.
\end{aligned}
\end{equation*}
Combining the above estimates, we conclude that the bound \eqref{eq:estimate_Wgammma} holds.

Moreover, for $|\vec{k}|+1 \le s$, we have
\[ \int_{\R^d} \int_{\R}\coeff (\vec{\omega} , -b)  D^{\vec{k}} \sigma(\vec{\omega} \cdot \x+b) db d\vec{\omega} = \int_{\R^d} \int_{\R}  \vec{\omega}^{\vec{k}} \coeff (\vec{\omega} , -b) \sigma^{(|\vec{k}|)}(\vec{\omega} \cdot \x+b) db d\vec{\omega},
\]
where $\vec{\omega}^{\vec{k}}:=\prod_{j=1}^d\omega_j^{k_j}$, and $\omega_j$ and $k_j$ denote the $j$-th components of $\vec{\omega}$ and $\vec{k}$, respectively. The proof is completed by applying estimate~\eqref{eq:estimate_Wgammma} once more together with the boundedness of $ \sigma^{(|\vec{k}|)}$.
\end{proof}

\begin{lemma}
    \label{lem:truncation_estimate}
Under the notation and assumptions of \cref{thm:integral_expression}, let $q>0$ and $Q:=[-q,q]$. For any $\generalfunc \in B^4(\Omega)$ and any $\vec{k}\in\mathbb{N}^d$ with $|\vec{k}|\le 2$, we have
\begin{equation*}
\left| D^{\vec{k}} \generalfunc(\x ) - \mathcal{T}_{\vec{k}}(\x) \right|
\lesssim \frac{\|\generalfunc\|_{B^4(\Omega)}}{q^{3 - |\vec{k}|}},
\quad \forall \x  \in \Omega,
\end{equation*}
where $\mathcal{T}_{\vec{k}}(\x) = \int_{Q^d \times Q} \coeff(\vec{\omega},-b)\,D^{\vec{k}} \sigma(\vec{\omega} \cdot \x + b)\,db\, d\vec{\omega}$.
\end{lemma}
\begin{proof}
Fix $q>0$ and set $Q:=[-q,q]$. Due to the boundedness of $\sigma$, we have
\begin{align*}
&\left|\int_{\R^d\times \R} \sigma (\vec{\omega}  \cdot \x +b)\coeff (\vec{\omega}  , -b)db d\vec{\omega}  - \int_{Q^d\times Q} \sigma (\vec{\omega}  \cdot \x +b)\coeff (\vec{\omega}  , -b)db d\vec{\omega}  \right| \\
\lesssim& \int_{\R^d} \int_{\R\setminus Q}|\coeff (\vec{\omega}  ,b)|db d\vec{\omega}  + \sum_{i=1}^d \int_{\R}\int_{\R^{d-1}}\int_{\R\setminus Q} |\coeff (\vec{\omega}  ,b)| d \omega_i d \omega_1 \cdots d \omega_{i-1} d \omega_{i+1} \cdots d \omega_d db
\\
\lesssim& \int_{\R^d} \int_{\R}\left|\frac{b^3}{q^3} \coeff(\vec{\omega}  , b)\right|  d bd\vec{\omega}  + \sum_{i=1}^d \int_{\R}\int_{\R^{d}} \frac{| \omega_i|^3}{q^3}|\coeff (\vec{\omega}  ,b)| d\vec{\omega}  db
\lesssim \frac{\norm{\generalfunc}_{B^4(\Omega)}}{q^3},
\end{align*}
where the last inequality uses the estimate (\ref{eq:estimate_Wgammma}). Similarly, for $|\vec{k}|=1, 2$, we conclude
\[
\left| D^{\vec{k}} \generalfunc(\x)   - \mathcal{T}_{\vec{k}}(\x)  \right| \lesssim \frac{\norm{\generalfunc}_{B^4(\Omega)}}{q^{3-|\vec{k}|}}.
\]
\end{proof}

\begin{remark}\label{remark:high_order_barron}
The result in \cref{lem:truncation_estimate} extends to higher-order Barron regularity. Specifically, for any $\generalfunc \in B^s(\Omega)$ with $s \geq 4$, we have
\[
\left| D^{\vec{k}} \generalfunc(\x) - \mathcal{T}_{\vec{k}}(\x) \right|
\;\lesssim\;
\frac{\|\generalfunc\|_{B^s(\Omega)}}{q^{s-1 - |\vec{k}|}},
\quad  |\vec{k}| \le 2, \quad \x \in \Omega.
\]
\end{remark}

Next, we estimate the expected error between $\mathcal{T}_{\vec{k}}(\x)$ and its discrete sum.
\begin{lemma}
\label{lem:discrete_error_of_appro}
Under the notation and assumptions of \cref{lem:truncation_estimate}, let $q, R>0$ and $M\in\mathbb{N}^+$. Assume that $\Omega \subset [-R,R]^d$. Let $Q=[-q,q]$, and let \(\{(\vec{\omega}_m,b_m)\}_{m=1}^M\) be i.i.d.\ samples drawn from a distribution with strictly positive density $\pi_{\vec{\omega},b}$ supported on \(Q^d \times Q\). For any \(\generalfunc \in B^4(\Omega)\),
define
\(
\tilde a_m = \coeff(\vec{\omega}_m,-b_m)/{\pi_{\vec{\omega},b}(\vec{\omega}_m,b_m)},\ m=1,\dots,M.
\)
Then for any multi-index \(\vec{k}\in\mathbb{N}^d\) with \(|\vec{k}|\le 2\), the following estimate holds:
\begin{equation*}
\mathbb{E}_{\{\vec{\omega},b\}}\!\left[
   \sup_{\x\in\Omega}
   \left|
   \frac{1}{M}\sum_{m=1}^M D^{\vec{k}}\sigma(\vec{\omega}_m\cdot\x+b_m)\,\tilde a_m
   - \mathcal{T}_{\vec{k}}(\x)
   \right|
   \right]
\lesssim \sqrt{\frac{I_q}{M}},
\end{equation*}
and the coefficients \(\{\tilde a_m\}_{m=1}^M\) satisfy the moment bound
\(
\frac{1}{M}\sum_{m=1}^M  (1+\|\vec{\omega}_m\|)^2\,|\tilde a_m| \leq C_q.
\)
Here, \( C_q, I_q<\infty\) are constants defined by the density $\pi_{\vec{\omega},b}$.
In particular, there exists a choice of \(\pi_{\vec{\omega},b}\) for which \( C_q, I_q\) are independent of \(q\).
\end{lemma}

\begin{proof}
For any multi-index $\vec{k} \in \mathbb{N}^d$ with $|\vec{k}|\le 2$, define $U_{m,\x} = D^{\vec{k}}\sigma (\vec{\omega}_m \cdot \x + b_m ) \tilde{a}_m$.
Then we have
\begin{equation*}
\begin{aligned}
\mathbb{E}&\left[ \sup_{\x\in \Omega} \left| \sum_{m=1}^M \frac{D^{\vec{k}}\sigma (\vec{\omega}_m \cdot \x + b_m )\tilde{a}_m}{M}
- \mathcal{T}_{\vec{k}}(\x)\right| \right]  = \mathbb{E}\left[ \sup_{\x\in \Omega} \left| \sum_{m=1}^M\frac{\left( U_{m,\x}- \mathbb{E}\left[U_{m,\x}\right] \right)}{M} \right| \right].
\end{aligned}
\end{equation*}
Let $\vec{\tau} = \{\tau_m\}_{m=1}^M$, where the $\tau_m$ are i.i.d. Rademacher random variables satisfying $\mathbb{P}(\tau_m=1)=\mathbb{P}(\tau_m=-1)=\frac{1}{2}$, independent of $\{(\vec{\omega}_m,b_m)\}_{m=1}^M$. By \cref{lem:symmetrization},
\begin{equation*}
 \mathbb{E}_{\{\vec{\omega},b\}}\left[ \sup_{\x\in \Omega} \left| \frac{1}{M}\sum_{m=1}^M\left( U_{m,\x}- \mathbb{E}\left[U_{m,\x}\right] \right) \right| \right]\leq 2 \mathbb{E}_{\vec{\tau},\{\vec{\omega},b\}}\left[ \sup_{\x\in \Omega} \left| \frac{1}{M}\sum_{m=1}^M \tau_m U_{m,\x} \right| \right].
\end{equation*}

We first consider the case $|\vec{k}|= 0$. Define $\varphi_m(t) = \sigma(t+b_m )-\sigma( b_m )$. Then
\begin{align*}
& \mathbb{E}_{\vec{\tau},\{\vec{\omega},b\}}\left[ \sup_{\x\in \Omega} \left| \frac{1}{M}\sum_{m=1}^M \tau_m U_{m,\x} \right| \right]
\leq \mathbb{E}_{\vec{\tau} , \{\vec{\omega} , b \}}  \left[  \left| \frac{1}{M}\sum_{m=1}^M \tau_m  \tilde{a}_m   \sigma(b_m ) \right|  \right] \\&\quad + \mathbb{E}_{\vec{\tau},\{\vec{\omega},b\}}  \left[ \sup_{\x\in \Omega} \left| \frac{1}{M}\sum_{m=1}^M \tau_m \tilde{a}_m \varphi_m (\vec{\omega}_m \cdot \x ) \right| \right]=: \text{\textcircled{1}} + \text{\textcircled{2}}.
\end{align*}
Let $L, B$ denote the uniform Lipschitz constant and the uniform upper bound for $\sigma$ (cf.  \cref{ass:activation1}). By \cref{lem:ying2024accurate} and Jensen’s inequality, we have
\[
    \text{\textcircled{2}} \leq 2L \mathbb{E}_{\vec{\tau},\{\vec{\omega},b\}} \left[ \sup_{\x\in \Omega} \left|  \sum_{m=1}^M \frac{\tau_m \tilde{a}_m \vec{\omega}_m \cdot \x}{M} \right| \right] \leq \frac{2LR\sqrt{d}}{M}\mathbb{E}_{\vec{\tau},\{\vec{\omega},b\}} \left[ \left\|\sum_{m=1}^M \tau_m \tilde{a}_m \vec{\omega}_m \right\|^2 \right]^{\frac{1}{2}}.
\]
Using the fact that $\mathbb{E}[\tau_i\tau_j]=\delta_{ij}$ and that $\{(\vec{\omega}_m, b_m)\}_{m=1}^M$ are i.i.d., we obtain
\begin{equation*}
\begin{aligned}
\text{\textcircled{2}}
 \leq \frac{2LR\sqrt{d}}{M}\mathbb{E}_{\{\vec{\omega},b\}}\left[ \left(\sum_{m=1}^M \left \|  \tilde{a}_m \vec{\omega}_m \right\|^2 \right)^{\frac{1}{2}}  \right]
 \leq \frac{2LR\sqrt{d}}{\sqrt{M}}\mathbb{E}_{\{\vec{\omega},b\}}\left[  \left\| \tilde{a}_1 \vec{\omega}_1 \right\|^2  \right]^{\frac{1}{2}}.
\end{aligned}
\end{equation*}
Similarly, using the boundedness of $\sigma$, we derive
\(\text{\textcircled{1}} \leq \frac{B}{\sqrt{M}}\mathbb{E}_{\{\vec{\omega},b\}}\left[  |\tilde{a}_1|^2    \right]^{\frac{1}{2}}\).
Such estimations also holds for $|\vec{k}|=1,2$ by defining $\varphi_m(t) = \sigma^{(|\vec{k}|)}(t+b_m )-\sigma^{(|\vec{k}|)}( b_m )$, where $ D^{\vec{k}}\sigma (\vec{\omega}_m \cdot \x + b_m )=\vec{\omega}_m^{\vec{k}}\left(\varphi_m(\vec{\omega}_m\cdot\x)+\sigma^{(|\vec{k}|)}( b_m )\right)$. Thus we conclude that
\begin{equation*}
    \mathbb{E}_{\{\vec{\omega},b\}}\left[ \sup_{\x\in \Omega} \left| \frac{1}{M}\sum_{m=1}^M\left( U_{m,\x}- \mathbb{E}\left[U_{m,\x}\right] \right) \right| \right] \leq  \frac{(4LR\sqrt{d}+2B)\sqrt{I_q}}{\sqrt{M}},
\end{equation*}
where $I_q=\max_{0\leq|\vec{k}|\leq 2}I_{\vec{k}}$ and
\[
\begin{aligned}
I_{\vec{k}}:=&\int _{Q^d\times Q}\frac{ \vec{\omega}^{2\vec{k}} \coeff^2 (\vec{\omega},-b)}{\pi_{\vec{\omega}, b}} ( \|\vec{\omega}\|^2 +1) db d\vec{\omega}\\
\leq &\sup_{Q^d\times Q}\frac{|\vec{\omega}^{\vec{k}}\coeff (\vec{\omega}, -b)|\cdot (\|\vec{\omega}\|+1)}{\pi_{\vec{\omega}, b}}
 \int _{\R^d\times \R} |\vec{\omega}^{\vec{k}} \coeff (\vec{\omega}, -b)| \cdot (\|\vec{\omega}\|+1) db d\vec{\omega} \\
 \lesssim & \sup_{Q^d\times Q}\frac{|\vec{\omega}^{\vec{k}} \coeff (\vec{\omega},-b)|\cdot (\|\vec{\omega}\|+1)}{\pi_{\vec{\omega}, b}}<\infty.
\end{aligned}
\]
Here, we have used the estimate \eqref{eq:estimate_Wgammma} and the fact that $|\coeff(\vec{\omega} , b)|\lesssim \norm{\generalfunc}_{B^4(\Omega)}$. In addition, we have
\[
\sum_{m=1}^M \frac{(\|\vec{\omega}_m\|+1)^2|\tilde{a}_m|}{M} \leq \sup_{Q^d\times Q}
\frac{(\|\vec{\omega}\|+1)^2|\coeff(\vec{\omega},-b)|}
     {\pi_{\vec{\omega},b}}: = C_q < \infty.
\]
In particular, by taking $\pi_{\vec{\omega}, b}= \frac{ |\coeff (\vec{\omega},-b)|\cdot(\|\vec{\omega}\|+1)^3}{ \norm{ |\coeff (\vec{\omega},-b)|\cdot (\|\vec{\omega}\|+1)^3}_{L^1(Q^d\times Q)}}$, the constants $ C_q, I_q$ become independent of $q$.
\end{proof}

With the above two lemmas, we are able to present the proof of \cref{thm:approximation_thm}.

\begin{proof}[Proof of \cref{thm:approximation_thm}]
Setting $a_m = \tilde{a}_m/M$ for $m = 1,\cdots, M$, where $\tilde{a}_m$ is defined in \cref{lem:discrete_error_of_appro}, the result follows from \cref{thm:integral_expression} and Lemmas \ref{lem:truncation_estimate} and \ref{lem:discrete_error_of_appro}.
\end{proof}

\subsection{Generalization error}\label{sec:genera}
We derive the generalization error in this section and begin by defining the Rademacher complexity. For the reader's convenience, the supporting theorems used in our derivation are provided in Appendix~\ref{app:Auxiliary Definitions and Lemmas}.

\begin{definition}[Rademacher complexity]
Let $\mu$ be a probability distribution on a set $\mathcal{X}$ and suppose that $x_1, \cdots, x_n$ are independent samples selected according to $\mu$. Let $V$ be a class of functions mapping from $\mathcal{X}$ to $\R$. The empirical Rademacher complexity of $V$ is defined as
\[
\widehat{R}_n(V) = \bE_{\vec{\tau}} \left[ \sup_{v \in V} \left| \frac{1}{n} \sum_{i=1}^n \tau_i v(x_i) \right| \,\bigg|\, x_1, \cdots, x_n \right],
\]
where $\vec{\tau} = \{\tau_i\}_{i=1}^n$, and $\tau_i$ are i.i.d. Rademacher random variable satisfying $\mathbb{P}(\tau_i=1)=\mathbb{P}(\tau_i=-1)=\frac{1}{2}$. The population Rademacher complexity of $V$ is then defined as the expectation of the empirical Rademacher
complexity over i.i.d. draws $x_1, \cdots, x_n$,
\begin{equation*}
 R_n (V) = \bE_{\{x\}, \vec{\tau}} \left[\sup_{v\in V} \left| \frac{1}{n}\sum_{j=1}^n \tau_j v(x_j) \right| \; \right].
\end{equation*}
\end{definition}

\begin{definition}
A function $\phi: [0,\infty) \rightarrow [0,\infty)$ is sub-root if it is nonnegative, nondecreasing and $r\rightarrow \phi(r)/\sqrt{r}$ is nonincreasing for $r > 0$.
\end{definition}

Next, we demonstrate that the generalization error can be bounded, provided that an appropriate sub-root function exists.

\begin{lemma}\label{lem:gen}
Consider the hypothesis space \( F(\Omega)\).
For any $r>0$, define the localized function classes
\(
\mM_{r}(\Omega)
:=\left\{u \in F(\Omega) : \ L(u)\le r \right\}
\)
and the associated vector-valued space
$\mS_{r}(\Omega)
:= \left\{ \left(\ell_{1}, \cdots, \ell_{5} \right)  : \  u \in \mM_{r}(\Omega) \right\}.$  With respect to the random sampling of the datasets~\eqref{eq:training_datasets}, assume that there exists a sub-root function $\phi(r)$ such that, for all $r \ge  r^*$,
\begin{equation}\label{ass:phi}
R_{\mathbf{n}}\left(\mS_{r}(\Omega)\right) :=
\sum_{i=1}^5 R_{n_i}\left(\left\{\ell_{i} : (\ell_{1},\cdots,\ell_{5}) \in \mS_{r}(\Omega) \right\}\right)
\ \leq\ \phi(r),
\end{equation}
where $ r^*$ is the unique solution of the fixed-point equation $r= \phi(r)$. Then, for any $t>0$, with probability at least $(1-e^{-t})^5$, the following bound holds:
\[
L(\hat \generalNN) \ \lesssim \ \inf_{ \generalNN \in F(\Omega)}   \| \generalNN - \generalfunc^* \|^2_{X^{\infty}( \Omega)} +  \max\left\{ r^* ,C_B^2\frac{t}{n }\right\},
\]
where $\hat \generalNN$ and $\solufunc$ are defined as in \cref{thm:main}.
\end{lemma}

\begin{proof}
For any $r\geq  r^* $, we introduce the normalized function classes
\begin{equation}\label{eq:ell}
\tilde{\ell}_{i}(\x) = \frac{\mathbb{E} [\ell_i]-\ell_i(\x)}{L(u)+ r},\quad \tilde{\mS}_{i}^r := \left\{ \ \tilde{\ell}_{i}:\ u \in F(\Omega) \right\},  \quad i=1, \cdots, 5.
\end{equation}
We first verify that these classes satisfy the conditions required for applying the improved version of Talagrand’s concentration inequality (\cref{lem:talagrand_ineq}).
Since all functions involved are continuous and the subdomains are compact, there exists a constant \( 0< \hat{C} \lesssim C_B \) such that
\begin{equation}
\label{eq:boundedness_chat}
\max\left\{\begin{array}{l}
 \sup_{\substack{ |\vec{k}|=0,1,2 \\ i=1,2}} \|D^{\vec{k}} \generalfunc^* \|_{C(\overline{\Omega}_i)},
\sup_{i=1,2},\
\|g_i\|_{C(\overline{\Omega}_i)},\
\|g_d\|_{C(\Gamma)},\\
\|g_b\|_{C(\partial\Omega)},\
\|g_n\|_{C(\Gamma)},\ \sup_{\substack{ |\vec{k}|=0,1,2 \\ i=1,2}} \|D^{\vec{k}} u\|_{C(\overline{\Omega}_i)}
   \end{array}
\right\} \leq \hat{C}.
\end{equation}
Since $L(u)=\sum_{i=1}^5 \lambda_i^2\bE [\ell_{i}]\ge\lambda_1^2\bE [\ell_{1}] \ge0$,
it follows that every $\tilde{\ell}_{1}\in \tilde{\mS}^r_{1}$ has zero mean, i.e.,
$\bE[\tilde{\ell}_{1}]=0$. Next, the uniform boundedness of $\tilde{\ell}_{1}$ can be estimated as
\begin{equation*}
 \left \|\tilde{\ell}_{1}\right\|_{\infty} \leq \frac{2\|\ell_{1}\|_{\infty}}{r} = \frac{2\left\|\left( \beta_1 \Delta u_1 + g_1 \right)^2\right\|_{\infty}}{r}
 \leq \frac{2(\beta_1+1)^2 \hat{C}^2}{r}:= \varrho_1.
\end{equation*}
In addition, since $\mathbb{E} [\ell_{1}^2] = \int_{\Omega_1} \frac{1}{|\Omega_1|} (\beta_1\Delta u_1+g_1)^4dx \leq
(\beta_1+1)^2 \hat{C}^2\bE [\ell_{1}]$ for all $u\in F(\Omega)$, we obtain the following bound for the variance:
\begin{align*}
\mathrm{Var}[\tilde{\ell}_{1}]
\leq \frac{\bE [\ell_{1}^2]}{2rL(u)}
\leq \frac{(\beta_1+1)^2 \hat{C}^2\bE [\ell_{1}]}{2rL(u)}
\leq \frac{(\beta_1+1)^2 \hat{C}^2}{2r\lambda_1^2} :=\zeta_1^2.
\end{align*}
The same bounds hold for all $\tilde{\mS}_{i}^r$, and the constants are denoted by $\varrho_i$ and $\zeta_i^2$.

For any $u\in F(\Omega)$ and $r\geq r^*$, let $\tilde{\ell}_{1},\cdots,\tilde{\ell}_{5}$ be the associated functions defined by \eqref{eq:ell}.
Then, for any $t>0$, applying the AM-QM inequality together with \cref{lem:talagrand_ineq} to each $\tilde{\ell}_i$  yields, with probability at least $(1-e^{-t})^5$,
\begin{equation}
\label{eq:sum_tilde_s}
\sum_{i=1}^5 \frac{\lambda_i^2}{n_i}\sum_{j=1}^{n_i}\tilde{\ell}_{i}(\x_{i,j})
\leq \sum_{i=1}^5 \lambda_i^2\left(2\mathbb{E}\left[\sup_{\tilde{\ell}_{i} \in \tilde{\mS}_{i}^r}\frac{1}{n_i}\sum_{j=1}^{n_i}\tilde{\ell}_{i}(\x_{i,j})\right]+ \sqrt{\frac{2t\zeta_{i}^2}{n_i}}+\frac{4t\varrho_{i}}{3n_i}\right).
\end{equation}
Under assumption~\eqref{ass:phi}, \cref{lem:symmetrization} can be invoked to bound the first term on the right-hand side:
\begin{equation*}
\resizebox{\linewidth}{!}{$
\begin{aligned}
&\sum_{i=1}^5 \bE \left[\sup_{\tilde{\ell}_{i}\in \tilde{\mS}_{i}^r}\frac{1}{n_i}\sum_{j=1}^{n_i}\tilde{\ell}_{i}(\x_{i,j})\right]
 \leq \sum_{i=1}^5  \bE\left[\sup_{u\in F(\Omega)}\left| \frac{1}{n_i}\sum_{j=1}^{n_i}\frac{\ell_{i}(\x_{i,j})-\mathbb{E} [\ell_{i}]}{L(u)+ r}\right| \right] \\
\leq&  2 \sum_{i=1}^5 \left[  R_{n_i}\left(\left\{\frac{\ell_{i} }{L(u)+r}: u \in \mM_r  \right\}\right)+\sum_{k=0}^{\infty} R_{n_i}\left(\left\{\frac{\ell_{i} }{L(u)+r}: u \in \mM_{4^{k+1}r}\setminus\mM_{4^{k}r} \right\}\right)
\right]  \\
\leq& 2\left[\frac{R_{\mathbf{n}}\left(\mS_{r}(\Omega)\right)}{r} + \sum_{k=0}^{\infty} \frac{R_{\mathbf{n}}\left(\mS_{4^{k+1}r}(\Omega)\right)}{r4^k+r} \right] \leq 2\left[ \frac{\phi(r)}{r} + \sum_{k=0}^{\infty} \frac{\phi(r4^{k+1})}{r4^k} \right]
\leq \frac{10\phi(r)}{r}.
\end{aligned}
$}
\end{equation*}
Let $n_{\min}=\min\{n_1,\cdots,n_5\}$. Combining this estimate with inequality~\eqref{eq:sum_tilde_s} yields
\begin{equation}
\label{eq:tilde_l_sum}
 \begin{aligned}
\frac{L(u) - L^{\mathbf{n}}(u) }{L(u) +r} = \sum_{i=1}^5 \frac{\lambda_i^2}{n_i}\sum_{j=1}^{n_i}\tilde{\ell}_{i}(\x_{i,j})
  \leq
 \frac{ M_1\phi(r)}{r}+ \sqrt{\frac{M_2t}{n_{\min}r}} +\frac{M_3t}{n_{\min}r} := \eta(r),
 \end{aligned}
\end{equation}
where the constants are defined as $M_1 = 20\max \big\{ \lambda_1^2, \cdots, \lambda_5^2 \big\}$, $M_3 =\frac{8}{15}M_2$, and \[M_2 = 5 \hat{C}^2\max\big \{\lambda_1^2(\beta_1+1)^2, \lambda_2^2(\beta_2+1)^2, \lambda_3^2(\beta_1+\beta_2+1)^2, 9\lambda_4^2, 4\lambda_5^2\big\}.\]

Let $r_0 = \max\left\{2^{6+2c^{\circ}} r^*,  \frac{36M_2t}{n_{\min} }\right\},$ where $c^{\circ}$ is the small number in $\mathbb{N}^+$ such that $M_1\leq 2^{c^{\circ}}$. It can be verified that $\eta(r_0) < \frac{1}{2}$, since
\[
\frac{M_1\phi(r_0)}{r_0} \leq \frac{2^{c^{\circ}}\phi(2^{6+2c^{\circ}} r^*)}{2^{6+2c^{\circ}} r^*} \leq\frac{1}{8} , \quad \sqrt{\frac{M_2t}{n_{\min} r_0}} \leq \frac{1}{6}, \quad
\frac{M_3t}{n_{\min} r_0} \le \frac{2}{135}.
\]
By setting $r=r_0$, it follows from inequality~\eqref{eq:tilde_l_sum} that, for any $u \in F(\Omega)$,
\begin{align*}
\frac{L(u) - L^{\mathbf{n}}(u) }{L(u) +r_0}   \leq \sum_{i=1}^5 \sup_{\tilde{\ell}_{i}\in \tilde{\mS}_{i}^r} \frac{\lambda_i^2}{n_i}\sum_{j=1}^{n_i}\tilde{\ell}_{i}(\x_{i,j}) \leq \eta(r_0) < \frac{1}{2}.
\end{align*}
Recalling that \(\hat{C} \lesssim C_B \) and $n_i = \mathcal{O}(n)$ for $i=1,\cdots,5$, and substituting $u=\hat u$ into the above inequality, we obtain
\[
L(\hat u) \leq 2L^{\mathbf{n}}(\hat u) + r_0 \lesssim \inf_{u \in F(\Omega)}   \|u - \generalfunc^* \|^2_{X^{\infty}( \Omega)} +  \max\left\{ r^* , C_B^2\frac{t}{n }\right\}.
\]
Here, we have used the fact that, under the regularity assumption \(\generalfunc^*  \in B^4(\Omega_1) \cap B^4(\Omega_2)\) (hence \(\generalfunc^*  \in C^2(\overline{\Omega}_1) \cap C^2(\overline{\Omega}_2)\)), there exists a constant \(\tilde{C} > 0\), depending only on \(\{\lambda_i\}_{i=1}^5\) and \(\beta\), such that
\[
L^{\mathbf{n}}(\hat \generalNN) \leq L^{\mathbf{n}}(\generalNN) \leq \tilde{C} \|\generalNN - \generalfunc^* \|^2_{X^{\infty}( \Omega)}, \quad \forall \generalNN \in F(\Omega).
\]
\end{proof}

Building upon the previous lemma, it remains to determine a suitable sub-root function $\phi$.
To this end, we need to bound the covering numbers of RaNNs.

\begin{definition}[Covering number]
 Let $(E,\rho)$ be a metric space with metric $\rho$. A $\delta$-cover of a set $\mathcal{X} \subset E$ with respect to $\rho$ is a collection of points $\{x_1,\cdots, x_n\}\subset \mathcal{X}$ such that for every $x\in \mathcal{X}$, there exists $i\in \{1,\cdots,n\}$ such that $\rho(x,x_i)\leq \delta$. The $\delta$-covering number $\mN(\delta,\mathcal{X},\rho)$ is the cardinality of the smallest $\delta$-cover of the set $\mathcal{X}$ with respect to the metric $\rho$. Equivalently, the $\delta$-covering number
 $\mN(\delta,\mathcal{X},\rho)$ is the minimal number of balls $B_{\rho}(x,\delta)$ of radius $\delta$ needed to cover $\mathcal{X}$.
\end{definition}

\begin{lemma}\label{lem:covering_number_tanh_nn}
Let
$G = \Sigma_{M,q,C_B}^{\tanh}(\Omega)$ denote the class of shallow RaNNs defined in \eqref{def:stable_shallow_rnn_space} with activation $\tanh$. Given a fixed set of sampling points $\{\x_1, \dots, \x_n\} \subset \Omega$, for any $\epsilon>0$, the $\epsilon$-covering number of the derivative class $D^{\vec{k}}G$ with respect to the empirical $L^1$-norm $\|\cdot\|_{n,1}$ satisfies
\[
\mathcal{N}\left(\epsilon, D^{\vec{k}}G, \|\cdot\|_{n,1}\right)\leq \left(\frac{2q^{|\vec{k}|} C_BM}{\epsilon}\right)^{M}, \quad \forall |\vec{k}| \leq 2.
\]
\end{lemma}
\begin{proof}
The proof is provided in Appendix~\ref{proof:covering number}.
\end{proof}

We are now in a position to give the proof of \cref{thm:gen}.

\begin{proof}[Proof of \cref{thm:gen}]

We first construct a sub-root function $\phi(r)$ for which condition \eqref{ass:phi} holds.
For any $u \in \mathcal{M}_r(\Omega)$, there exists a constant $C_0>1$, depending only on $\{\lambda_i\}_{i=1}^5$ and $\beta$, such that
\begin{equation*}
 \max\left\{
   \begin{array}{l}
   \mathbb{E}_{\Omega_1}\left[(\Delta [u_1-\generalfunc^* _1])^2\right] ,\ \mathbb{E}_{\Omega_2}\left[(\Delta [u_2-\generalfunc^* _2])^2\right], \  \mathbb{E}_{\Gamma}\left[(\llbracket  u- \generalfunc^*  \rrbracket)^2\right],  \\
\mathbb{E}_{\Gamma}\left[(\llbracket \beta \nabla [u- \generalfunc^* ]\cdot\mathbf{n} \rrbracket)^2\right] , \ \mathbb{E}_{\partial\Omega}\left[(u_2-\generalfunc^* _2)^2\right]
   \end{array}
   \right\}\leq C_0 r.
\end{equation*}
Moreover, let $C_1=2\hat{C}\max\{ 2 \beta_1(\beta_1+1), \  2\beta_2(\beta_2+1),\ 2(\beta_1+\beta_2+1),\ 6\}$ and let $(\ell_1,\dots,\ell_5)$ be the associated tuple in $\mathcal{S}_{r}(\Omega)$. By \cref{lem:Ledoux-Talagrand contraction}, we have
\begin{equation}
\begin{aligned}\label{eq:bounds_of_rns}
 &R_{n_i}\left(\ell_{i}(\x)\right) \leq C_1  R_{n_i}\left( \Delta [u_i- \generalfunc^* _i]:\, \mathbb{E}_{\Omega_1}\left[(\Delta [u_1-\generalfunc^* _1])^2\right]\leq C_0 r \right),\ i=1, 2,
\\
&R_{n_3}\left(\ell_{3}(\x)\right) \leq C_1  R_{n_3} \left( \llbracket \beta \nabla [u- \generalfunc^* ]\cdot\mathbf{n} \rrbracket : \mathbb{E}_{\Gamma}\left[(\llbracket \beta \nabla [u- \generalfunc^* ]\cdot\mathbf{n} \rrbracket)^2\right]\leq C_0 r \right),
\\
&R_{n_4}\left(\ell_{4}(\x)\right) \leq C_1  R_{n_4} \left( \llbracket  u- \generalfunc^*  \rrbracket: \mathbb{E}_{\Gamma}\left[(\llbracket  u- \generalfunc^*  \rrbracket)^2\right]\leq C_0 r \right), \\
& R_{n_5}\left(\ell_{5}(\x)\right) \leq C_1   R_{n_5} \left(u_2-\generalfunc^* _2:  \mathbb{E}_{\partial\Omega}\left[(u_2-\generalfunc^* _2)^2\right]\leq C_0 r \right).
\end{aligned}
\end{equation}
Here we have used the fact that the quadratic function is Lipschitz continuous on bounded domains, together with the existence of the uniform bound $\hat{C}$ from \eqref{eq:boundedness_chat}.
In addition, we introduce the following function classes associated with the residuals
\[
\begin{aligned}
&F_i = \left\{\Delta [u_i - \generalfunc_i^*] \big|_{\Omega_i}: u\in F(\Omega) \right\},\ i=1,2,
&&F_3 = \left\{\llbracket \beta \nabla [u- \generalfunc^* ]\cdot\mathbf{n} \rrbracket \big|_{\Gamma}:  u\in F(\Omega)\right\}, \\
&F_4 = \left\{\llbracket  u- \generalfunc^*  \rrbracket \big|_{\Gamma} : u\in F(\Omega)\right\}, \
&&F_5 = \left\{\left.(u_2 - \generalfunc_2^*)\right|_{\partial\Omega}: u\in F(\Omega)\right\}.
\end{aligned}
\]
Let
\(
\operatorname{star}(F, f^*)
:= \left\{ f^* + \xi(f-f^*): f \in F,\ \xi \in [0,1] \right\}
\) denote the star-hull of $F$ around $f^*$, we then define the auxiliary functions
\begin{equation}
\label{eq:phi_1}
   \psi_i(r) = 10b\, R_{n_i}\!\left\{ f \in \operatorname{star}(F_i, 0) : \mathbb{E}[f^2] \le r \right\} + \frac{11 b^2 \log n_i}{n_i}, \ i=1,\cdots,5,
\end{equation}
where $b = \max\{1, 3\hat{C}, 2d\hat{C}, (1+\beta_1+\beta_2)\hat{C}\}$.
We now construct the desired function as follows:
\[
\phi(r) = C_1  \sum_{i=1}^5 \psi_i(C_0 r).
\]
By \cref{lem: subroot_f_p2}, \(\phi(r)\) is a sub-root function that satisfies \(
R_{\mathbf{n}}\big(\mathcal{S}_{r}(\Omega)\big) \leq \phi(r)\). It suffices to estimate the bound of the fixed-point $ r^* $ where $\phi( r^* )= r^* $.

Next we derive sub-root upper bounds for $\psi_{i}$ for any $r \geq r^*$.  \Cref{lem: subroot_f_p1} implies that $r \geq r^*$ is equivalent to $r\geq \phi(r)$, and thus
\begin{equation}
\label{eq:phi_2}
C_0 r \geq r \geq \max\left\{\psi_1(C_0 r),\cdots,  \psi_5(C_0 r) \right \}.
\end{equation}
As an example, we derive a sub-root upper bound for $\psi_1(C_0 r)$.
By \cref{lem: subroot_f_p3}, with probability at least $1 - 1/n_1$, we have
\[
\left\{ f \in \operatorname{star}\left(F_1, 0\right): \mathbb{E}[f^2] \leq C_0 r \right\} \subset \left\{ f \in \operatorname{star}\left(F_1, 0\right): \|f\|_{n_1,2}^2 \leq 2C_0 r \right\},
\]
and thus
\[
R_{n_1} \left\{ f \in \operatorname{star}\left(F_1, 0\right): \mathbb{E}[f^2] \leq C_0 r \right\} \leq R_{n_1} \left\{ f \in \operatorname{star}\left(F_1, 0\right): \|f\|_{n_1,2}^2 \leq 2C_0 r \right\} + \frac{b}{n_1}.
\]
Observing the fact that $R_n\leq \sup \widehat{R}_{n}$, it suffices to estimate $\widehat{R}_{n}$. By \cref{thm:dudley},
\begin{equation}
\label{eq:f1_r}
\begin{aligned}
\widehat{R}_* :=&\widehat{R}_{n_1}\left\{ f \in \operatorname{star}\left(F_1, 0\right): \|f\|_{{n_1},2}^2 \leq 2C_0 r \right\}\\
\leq&\inf_{0\leq \delta\leq \sqrt{2C_0 r}} \left\{4\delta + \frac{12}{\sqrt{n_1}}\int_\delta^{\sqrt{2C_0 r}} \sqrt{\log \mN\left(\epsilon, \operatorname{star}\left(F_1, 0\right), \|\cdot\|_{{n_1},2}\right)} \,d\epsilon\right\}.
\end{aligned}
\end{equation}
Using Lemma 16.5 in \cite{gyorfi2006distribution} and $\|\cdot\|_{n,2}\leq\sqrt{n} \|\cdot\|_{n,1}$, for any $\epsilon>0$, we have
\[
\mN\left(\epsilon, \operatorname{star}\left(F_1, 0\right), \|\cdot\|_{{n_1},2}\right) \leq \frac{2b}{\epsilon} \mN\left(\frac{\epsilon}{2\sqrt{n_1}}, F_1, \|\cdot\|_{{n_1},1}\right).
\]
By applying \cref{lem:covering_number_tanh_nn} together with Lemmas 16.4 and 16.5 from \cite{gyorfi2006distribution}, we obtain
\begin{equation}\label{eq:cover_n_lu}
 \begin{aligned}
 &\mathcal{N}\left(\epsilon, F_1, \|\cdot\|_{n_1,1}\right) \leq \left(\frac{2d C_BM q^2}{\epsilon}\right)^{dM}.
\end{aligned}
\end{equation}
From \eqref{eq:phi_1} and \eqref{eq:phi_2}, we have $\log n_1/{n_{1}}\leq r$. We therefore set $\delta=1/{n_1}$ in \eqref{eq:f1_r}. Combining the covering number estimates above, we obtain:
\begin{equation*}
\begin{aligned}
  \widehat{R}_*
  &\lesssim \frac{4}{{n_1}} +\sqrt{\frac{dM C_0r}{{n_1}}\log\left(4dC_B M(q+1)^2  {n_1}^{3/2}\right)}.
\end{aligned}
\end{equation*}
Then $\psi_1(C_0r)$ has a sub-root upper bound:
\[
\psi_1(C_0r)
\lesssim 10 b \sqrt{\frac{dM C_0r}{{n_1}}\log\left(4dC_B M (q+1)^2  {n_1}^{3/2}\right)}
+ \frac{40b + 10 b^2 +11b^2 \log n_1}{n_1}.
\]
Similar derivations apply to the upper bounds for $\psi_i(C_0r), i=2,\cdots, 5$ and are thus omitted for brevity.

Let $n_{\min}=\min\{n_1,\cdots,n_5\}$ and $n_{\max}=\max\{n_1,\cdots,n_5\}$.   Recalling that $C_1\lesssim \hat{C}$, $b\lesssim \hat{C} $ and $\hat{C} \lesssim C_B$, and combining estimates \eqref{eq:bounds_of_rns}-\eqref{eq:cover_n_lu}, we deduce that
\[
\phi(r) \lesssim \sqrt{\frac{C_B^4d M C_0 r}{n_{\min} } \log\!\left( d C_B M (q+1) n_{\max}\right)}.
\]
Combining this estimate with \cref{lem: subroot_f_p1}, we conclude that the fixed point $r^*$ satisfies the following bound:
\[
r^* \lesssim\frac{ C_B^4 d M\log ( dC_BM(q+1)n_{\max})}{n_{\min} }.
\]

Substituting the estimate of $r^*$ into \cref{lem:gen} and choosing $t = \log (n_{\max})$, we find that  with probability at least $(1-1/n_{\max})^5$,
\begin{align*}
L(\hat \generalNN) \lesssim \inf_{\generalNN \in F(\Omega)}   \|\generalNN - \generalfunc^* \|^2_{X^{\infty}( \Omega)} +   \frac{C_B^4 d M \log (dC_BM(q+1)n_{\max})}{n_{\min} } .
\end{align*}
Since $L(u) \lesssim C_B^2$ and $1-(1-1/n)^5 < 5/n$ for all $n\in\mathbb{N}^+$, taking the expectation over the random sampling of the datasets~\eqref{eq:training_datasets}, we obtain
\[
\begin{aligned}
\mathbb{E}_{\{\x\}} \left[L(\hat \generalNN)\right] &\lesssim
\inf_{\generalNN \in F(\Omega)}   \|\generalNN - \generalfunc^* \|^2_{X^{\infty}( \Omega)} +  \frac{C_B^4dM \log ( dC_BM(q+1)n_{\max})}{n_{\min} } +  \frac{ 5 C_B^2}{n_{\max}} \\
&\lesssim  \inf_{\generalNN \in F(\Omega)}   \|\generalNN - \generalfunc^* \|^2_{X^{\infty}( \Omega)} +  \frac{C_B^4dM \log (dC_BM(q+1)n_{\max})}{n_{\min} }.
\end{aligned}
\]
Finally, for $n\in \mathbb{N}^+$, setting $n_i = \mathcal{O}(n)$ for $i=1,\cdots,5$ concludes the proof.
\end{proof}

\section{Numerical experiments}\label{sec:numerical_results}
This section presents several numerical examples to demonstrate the theoretical findings.
The overall setting in all experiments is summarized as follows:
\begin{itemize}
\item \textit{Network setting.} We use the RaNN hypothesis space $F(\Omega)$ defined in \cref{eq:hypothesis_space_rnn} for function approximation. The hidden-layer parameters are initialized from a truncated normal distribution (\texttt{scipy.stats.truncnorm}) supported on the interval $[-q,q]$, with $q=10$ by default.

\item \textit{Generation of data.} A total of $N$ collocation points are sampled, with $\lfloor N/2\rfloor$ drawn from the interior of $\Omega$, $\lfloor N/4\rfloor$ from the interface $\Gamma$, and $\lfloor N/4\rfloor$ from the boundary $\partial\Omega$.

\item \textit{Implementation.}
We solve the unconstrained variant of \eqref{eq:empirical_loss_function_of_rann} with \texttt{scipy.linalg.lstsq}; the prescribed constraints are sufficiently loose and the coefficients remain finite at the empirical solution. All experiments are performed in Python~3.8 on a workstation equipped with an Intel Xeon W-3335 CPU @3.40 GHz.
\end{itemize}
All reported results are averaged over ten independent runs, each using independently sampled training data and random initializations.

\hypertarget{example1}{\textbf{Example} 1.}
Let $\Omega=[-1,1]^2$, and define the interface curve by
\(
\Gamma := \{\x \in \Omega : v(\x)-\pi/6.28 = 0\},\) where
\(v(\x) = \sqrt{x_1^2+x_2^2}
\) and \(v_0=\pi/6.28\).
The exact solution is prescribed as follows:
\begin{equation*}
\generalfunc^*(\x) =
\begin{cases}
\dfrac{v(\x)^3}{\beta_1}, & \x \in \Omega_1, \\
\dfrac{v(\x)^3}{\beta_2}+\left(\dfrac{1}{\beta_1}-\dfrac{1}{\beta_2}\right)v_0^3,
& \x \in \Omega_2.
\end{cases}
\end{equation*}
The functions $g$, $g_d$, $g_n$, and $g_b$ are derived directly from $\generalfunc^*$ and coefficient $\beta$. This benchmark problem has been previously investigated in \cite{chu2010new, wang2020mesh}.

\begin{figure}[htbp]
 \centering
\includegraphics[width=0.7\linewidth]{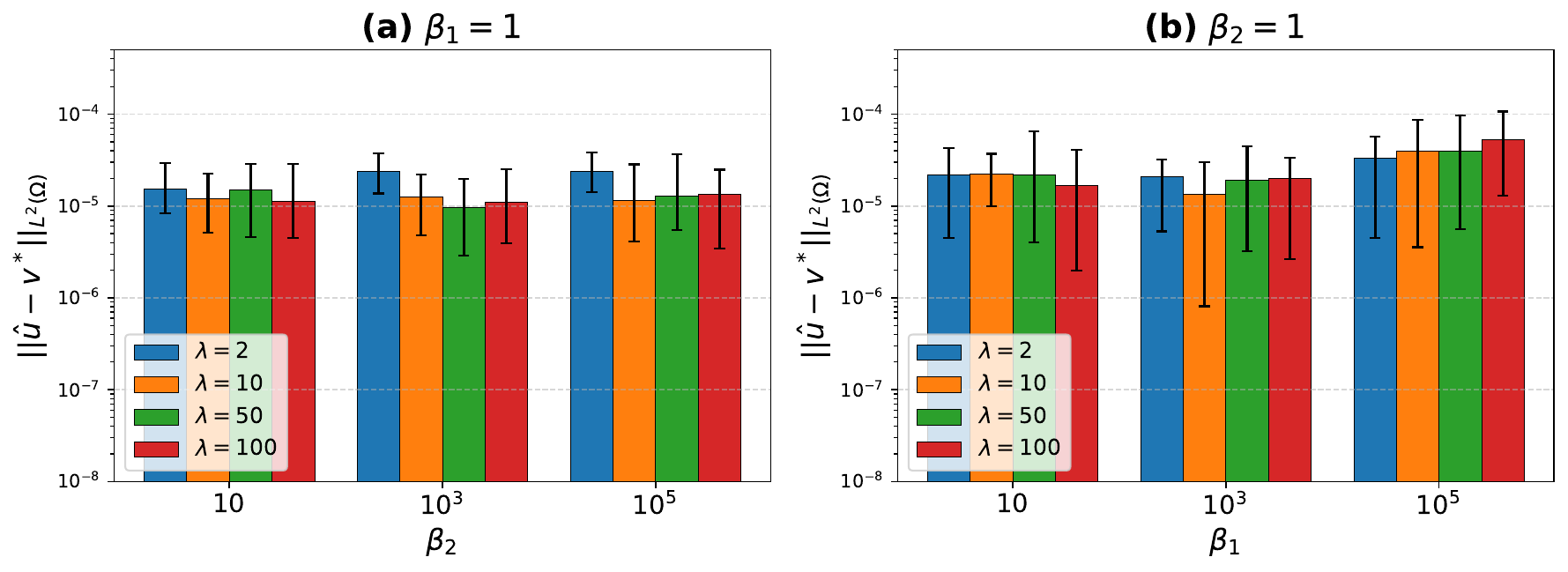}
 \caption{Mean \(L^2\)-norm errors in Example~\protect\hyperlink{example1}{1} for different coefficient contrasts and penalty weights \(\lambda=\lambda_4=\lambda_5\). In panel~(a), \(\beta_1=1\) and \(\beta_2\) varies; in panel~(b), \(\beta_2=1\) and \(\beta_1\) varies. Here, \(q=10\), \(\text{DOFs}=2400\), and \(N=16000\).}
 \label{fig:lam_ex1}
\end{figure}

To assess the sensitivity of the RaNN method to the choice of the contrast-independent penalty weights in \eqref{eq:penalty_weights}, we set \(\lambda_4=\lambda_5=\lambda\) and solve Example~\protect\hyperlink{example1}{1} using different values of \(\lambda\) under several coefficient contrasts. \Cref{fig:lam_ex1} reports the mean \(L^2\)-norm errors over ten independent runs, with error bars indicating the corresponding minimum and maximum values. When $\beta_1=1$, increasing \(\lambda\) from \(2\) improves the accuracy slightly, whereas the errors remain broadly comparable across the tested values of $\lambda$ when $\beta_2=1$. We therefore choose \(\lambda=10\) as a moderate choice that consistently yields competitive accuracy across the tested coefficient contrasts. This setting is used in all subsequent experiments.

\begin{table}[htbp]
\centering
\fontsize{7}{8}\selectfont
\begin{threeparttable}
\caption{The $L^2$-norm errors and $H^1$ semi-norm errors for different $\beta$ in Example \protect\hyperlink{example1} {1}.}
\begin{tabular}{cccccc}
\toprule
 \multicolumn{2}{c}{Coefficient $\beta$} &
 \multicolumn{2}{c}{\makecell{ MsFEM \cite{chu2010new} \\ ($h=1/64$)}} &
 \multicolumn{2}{c}{\makecell{ Ours \\ ($q=10, \text{DOFs}=2400, N=16000$)}}\cr
 \cmidrule(lr){3-4} \cmidrule(lr){5-6}
 & & $L^2$-norm & $H^1$ semi-norm & $L^2$-norm & $H^1$ semi-norm \cr
 \midrule
\multirow{3}{*}{$\beta_2=1$}& $\beta_1=10^1$
& 3.6619e-04 & 3.1374e-02 & 1.7027e-05 & 4.5565e-05
\\
& $\beta_1=10^3$
& 3.6482e-04 & 3.0915e-02 & 1.2205e-05 & 2.9530e-05
\\
&$\beta_1=10^5$
& 3.6366e-04 & 3.0662e-02 & 3.1100e-05 & 7.5262e-05
\cr
 \midrule
\multirow{3}{*}{$\beta_1=1$}& $\beta_2=10^1$
& 7.7605e-05 & 9.9543e-03 & 9.5298e-06&  2.9385e-04
\\
& $\beta_2=10^3$
& 5.4716e-05 & 6.2600e-03  & 1.0830e-05 & 2.3534e-04
\\
&$\beta_2=10^5$
& 5.9580e-05 & 6.2529e-03  & 1.0211e-05 & 2.5150e-04
\\
\bottomrule
\end{tabular}
\label{tab:table2_ex1}
\end{threeparttable}
\end{table}

\begin{figure}[htbp]
 \centering
\includegraphics[width=1\linewidth]{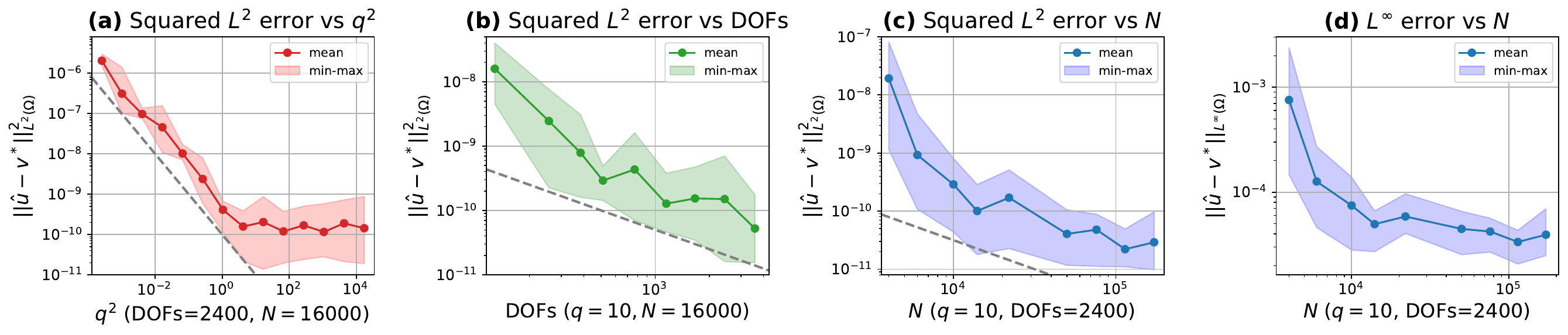}
 \caption{Errors for different $q$, DOFs, and $N$ in Example \protect\hyperlink{example1}{1}. The dashed lines in panels (a), (b), (c) indicate the reference convergence rates $\mathcal{O}(1/q^2)$, $\mathcal{O}(1/M)$,  and $\mathcal{O}(1/N)$, respectively.}
 \label{fig:f2_ex1}
\end{figure}

We next compare the RaNN method with the multiscale finite element method (MsFEM) \cite{chu2010new}; the results are reported in \cref{tab:table2_ex1}. In all test cases, RaNN attains smaller $L^2$-norm and $H^1$ semi-norm errors than those obtained by MsFEM.
Subsequently, to verify the error estimates in \cref{thm:main}, we consider Example~\protect\hyperlink{example1}{1} with $\beta_1=1$ and $\beta_2=10^3$. As shown in \cref{fig:f2_ex1}(a), the squared $L^2$ error initially decays at a rate broadly consistent with $\mathcal{O}(q^{-2})$, in agreement with the corresponding $q$-dependent term in \cref{thm:main}. For larger values of $q$, the error reaches a plateau, which is consistent with the truncation error becoming smaller than the remaining error components. The dependence on the number of degrees of freedom ($2M$, denoted by DOFs) is shown in \cref{fig:f2_ex1}(b). The numerical errors follow the anticipated $\mathcal{O}(1/M)$ rate. \Cref{fig:f2_ex1}(c) shows that the error decreases with the sample size $N$ at a rate broadly consistent with $\mathcal{O}(N^{-1})$. Furthermore, \cref{fig:f2_ex1}(d) shows that the $L^\infty$ error generally decreases with $N$ before reaching a plateau. Overall, these results are consistent with the convergence rates predicted by \cref{thm:main} and provide empirical support for the theoretical error estimate.

\hypertarget{example2}{\textbf{Example} 2.} This example examines the error behavior in the presence of a complex interface and high coefficient contrasts. Let $\Omega = [-1,1]^2$. We consider the interface problem \eqref{eq:interface_problem} with a highly oscillatory interface defined by
\(
 (x_1 - 0.02\sqrt{5})^2 + (x_2 - 0.02\sqrt{5})^2 = r^2(\theta),
\)
where $r(\theta) = 0.4 + 0.2\sin(20 \theta)$ for $\theta \in [0, 2\pi)$, as shown in  \cref{fig:ex2_fig1} (a).
The exact solution is prescribed as
\[
\generalfunc^*(\x) =
\begin{cases}
\frac{1}{\beta_1} e^{x_1 x_2}, & \x \in \Omega_1,\\
\frac{1}{\beta_2} \sin(\pi x_1) \sin(\pi x_2), & \x \in \Omega_2.
\end{cases}
\]
The source term, interface jump data, and boundary data
$g$, $g_d$, $g_n$, and $g_b$ are obtained by substituting
$\generalfunc^*$ into Eq. \eqref{eq:interface_problem}.

\begin{table}[htbp]
\centering
\fontsize{7}{8}\selectfont
\setlength{\tabcolsep}{4pt}
\renewcommand{\arraystretch}{1.15}
\begin{threeparttable}
\caption{Relative \(L^2\) errors for different values of \(\beta\) in Example~\protect\hyperlink{example2}{2}.}
\label{tab:ex2_beta}
\begin{tabular}{@{}ccccc@{}}
\toprule
\((\beta_1,\beta_2)\)
&
\makecell{LRaNN-FDM~\cite{li2025local} \\
\(\left(\begin{gathered}\mathrm{DOFs}=640\\ N=1920
\end{gathered}\right)\)}
&
\makecell{Mixed LRaNN-FDM~\cite{li2025local} \\
\(\left(\begin{gathered}
\mathrm{DOFs}=1920\\
N=1920
\end{gathered}\right)\)}
&
\makecell{Ablation \tnote{a} \\
\(\left(\begin{gathered}\mathrm{DOFs}=640 \\ N=1800\end{gathered}\right)\)}
&
\makecell{Ours \tnote{b} \\
\(\left(\begin{gathered}\mathrm{DOFs}=640 \\ N=1800\end{gathered}\right)\)}

\\
\midrule
 $(1, 10^2)$ & 1.75e-08 & 4.67e-09  &  1.64e-09  & 4.03e-10
\\
$(1, 10^4)$ & 1.44e-08 & 2.98e-08  &   1.90e-09 & 7.86e-12
 \\
$(10^{-4}, 10^4)$ & 1.04e-08  & 7.94e-09  & 1.14e-08    & 7.04e-12
\\
$(10^{-6}, 10^6)$ & 2.25e-07  & 3.16e-08 &  2.79e-07  &  4.53e-12
\\
$(10^2, 10^{-2})$ & 4.53e-08  & 7.06e-08  &  9.04e-07  &  2.23e-08
\\
$(10^4, 10^{-4})$ & 4.01e-06  & 1.60e-05 &  7.82e-04   &  2.97e-08
\\
\bottomrule
\end{tabular}
\begin{tablenotes}[flushleft]
\footnotesize
\item[a] The penalty weights are taken as in \cite{li2025local}:
\(\lambda_1=\lambda_2=1\) and
\(\lambda_3=\lambda_4=\lambda_5=100\).
\item[b] The penalty weights \(\lambda_i\) are chosen according to
Eq.~\eqref{eq:penalty_weights}.
\end{tablenotes}
\end{threeparttable}
\end{table}

This test problem coincides with Case~1 of Example~4.1 in \cite{li2025local}, allowing a direct comparison between the present formulation and the LRaNN-FDM method. Both methods employ local randomized single-hidden-layer networks and determine only the output-layer coefficients by linear least squares. The main differences are that LRaNN-FDM approximates differential operators by finite differences, uses empirically chosen penalty weights, and adopts problem-specific sampling ranges for the hidden-layer parameters. \Cref{tab:ex2_beta} reports the relative \(L^2\) errors and includes those of the mixed LRaNN-FDM variant for reference. The fixed-weight ablation performs less competitively than the reference methods in the last two cases, whereas the theoretically derived penalty weights yield improved accuracy in all cases.

\begin{figure}[htbp]
 \centering
\includegraphics[width=1\linewidth]{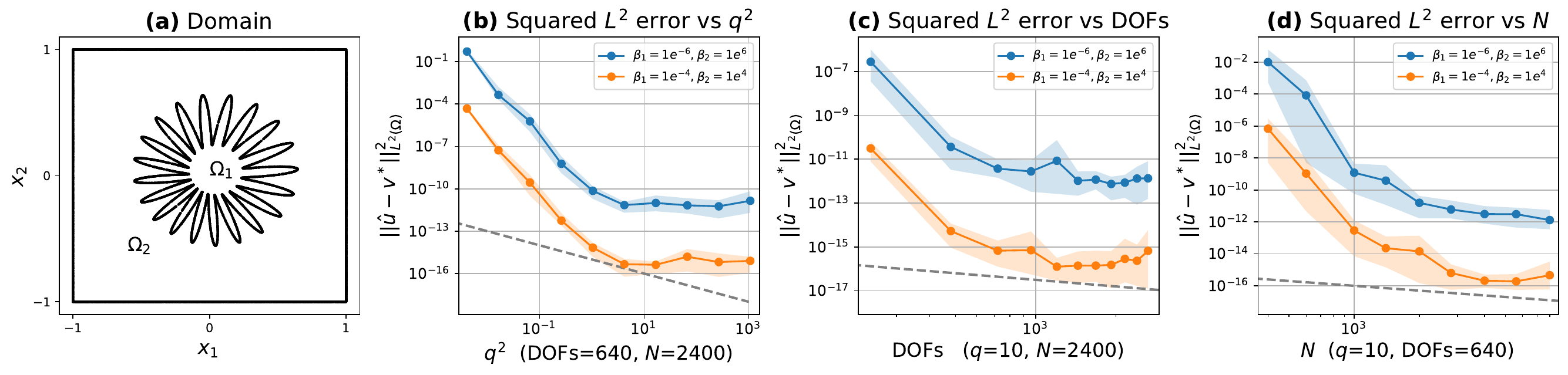}
 \caption{Errors for different values of $q$, DOFs, $N$ and $\beta$ in Example \protect\hyperlink{example2}{2}. The dashed lines in panels (b), (c), (d) indicate the reference convergence rates $\mathcal{O}(1/q^2)$, $\mathcal{O}(1/M)$,  and $\mathcal{O}(1/N)$, respectively.}
 \label{fig:ex2_fig1}
\end{figure}

\Cref{fig:ex2_fig1} shows the error behavior for $(\beta_1,\beta_2)=(10^{-4},10^4)$ and $(10^{-6},10^6)$, corresponding to coefficient contrasts of $10^8$ and $10^{12}$, respectively. As shown in panels~(b)-(d), the squared \(L^2\) errors exhibit similar convergence trends in the two cases, despite the increase in coefficient contrast.
The observed decay is even faster than that indicated by the reference slopes over certain parameter ranges. This faster decay may be attributed to the higher regularity of the exact solution $v^*$, which leads to more favorable approximation properties, as discussed in Remark~\ref{remark:high_order_barron}. A systematic study of this behavior will be pursued in future work.  Overall, the results indicate that the componentwise convergence behavior predicted by \cref{thm:main} persists in the presence of a highly oscillatory interface and high coefficient contrast.

\hypertarget{example3}{\textbf{Example} 3.}
To verify the theoretical estimates in three dimensions, we finally consider an interface problem from \cite{li2025local}.  The problem is posed on $\Omega=[-1,1]^3$ with the spherical interface
\(
\Gamma: x_1^2 + x_2^2 + x_3^2 = 0.75^2
\)
and the exact solution is prescribed as
\[
\generalfunc^*(\x)  =
\begin{cases}
5e^{x_1^2 + x_2^2 + x_3^2} + 20, & \x \in \Omega_1, \\
10(x_1 + x_2 + x_3), & \x \in \Omega_2.
\end{cases}
\]
The functions $g, g_d, g_n$ and $g_b$ are derived directly from $\generalfunc^*$ and the given coefficients.

\begin{table}[htbp]
\centering
\fontsize{7}{8}\selectfont
\begin{threeparttable}
\caption{$L^2$-norm errors for different $\beta$ in Example \protect\hyperlink{example3} {3}.}
\begin{tabular}{ccccccc}
\toprule
 $(\beta_1, \beta_2)$ &   $(1, 10)$  &   $(1, 10^2)$  & $(1, 10^3)$ \cr
 \midrule
\makecell{ LRaNN-FDM \cite{li2025local} \\ ($\text{DOFs}=2560, N=6075$)} &5.00e-04
&   1.05e-04 & 3.86e-05\\
 \midrule
\makecell{Immersed FEM \cite{li2025local} \\  ($\text{DOFs}=551368, h=1/81$)} & -
& 3.68e-03   & 1.36e-02  \\
 \midrule
\makecell{Body-fitted linear FEM \cite{li2025local} \\  ($\text{DOFs}=165202, h=1/30$)  }
& 1.65e-02 & 1.70e-02  & 1.70e-02\\
 \midrule
\makecell{Ours \\  ($\text{DOFs}=2560, N=6000$)  }
& 3.76e-07 & 5.37e-07  & 2.34e-06  \\
\bottomrule
\end{tabular}
\label{tab:tab1_ex3}
\end{threeparttable}
\end{table}

\begin{figure}[htbp]
 \centering
\includegraphics[width=1\linewidth]{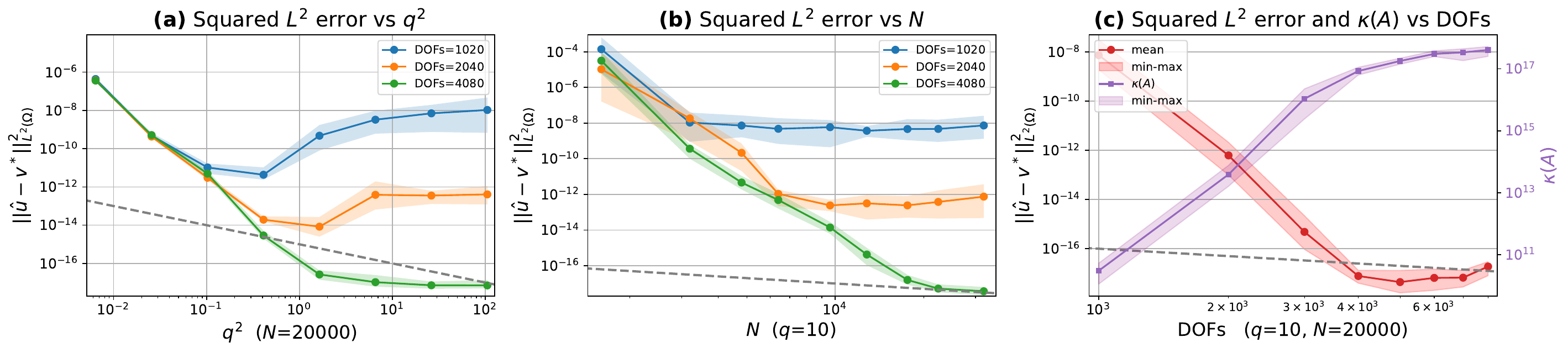}
 \caption{Errors for different values of $q$, $N$ and  DOFs in Example \protect\hyperlink{example3}{3}. The dashed lines in panels (a), (b), (c) indicate the reference convergence rates $\mathcal{O}(1/q^2)$, $\mathcal{O}(1/N)$,  and $\mathcal{O}(1/M)$, respectively.}
 \label{fig:fig1_ex3}
\end{figure}

The results in \cref{tab:tab1_ex3} demonstrate the accuracy of the proposed RaNN solver for the three-dimensional interface problem. It achieves $L^2$-norm errors up to four orders of magnitude smaller than the reported finite-element errors while using approximately two orders of magnitude fewer DOFs, and its accuracy remains stable under large coefficient contrasts. For the representative case with $\beta_1=1$ and $\beta_2=10^3$, \cref{fig:fig1_ex3} further shows the dependence of the error on $q$, $N$, and $M$. As $q$ or $N$ increases, the corresponding truncation or generalization error decreases until the approximation error associated with DOFs becomes dominant. Further increasing DOFs reduces the approximation error and improves accuracy; however, as shown in \cref{fig:fig1_ex3} (c), the accompanying growth in the condition number $\kappa(A)$ eventually amplifies floating-point errors, leading to an error plateau. These observations are consistent with the error decomposition underlying the theoretical estimate and further demonstrate the effectiveness of the RaNN method for three-dimensional interface problems.

\section{Conclusions}
\label{sec:conclusions}
This work develops a unified theoretical and numerical framework for solving elliptic interface problems using neural networks in a meshfree manner. From a theoretical viewpoint, we present the first systematic error analysis for neural network methods applied to elliptic interface problems, accounting for approximation, generalization, and optimization errors. More broadly, our approximation analysis carries over to a wide range of RaNN-based formulations, providing approximation guarantees without the restrictive assumptions frequently required in prior work. Meanwhile, our generalization analysis can be extended to other neural PDE solvers beyond RaNNs, leading to sharper generalization rates. Numerically, parameter choices guided by our theoretical analysis lead to consistent improvements on benchmark problems compared with heuristically tuned neural network methods and classical mesh-based finite element methods.
The experiments validate the theoretical error estimates and demonstrate the robustness of the proposed RaNN framework, even in the presence of high coefficient contrasts.

Several interesting directions remain to be explored. One direction is to extend the analysis to nonlinear and time-dependent PDEs, which would broaden its applicability and clarify the role of randomization in dynamical settings. Another is to develop adaptive randomization and sampling strategies, in which the parameter distribution is theoretically adjusted according to the PDE setting. In addition, extending the RaNN framework to operator learning  may provide a pathway for quantifying the total error decomposition in neural solvers applied to parametric PDEs.

\bibliographystyle{plain}
\bibliography{references}

\appendix

\section{Auxiliary Definitions and Lemmas}\label{app:Auxiliary Definitions and Lemmas}

For the reader's convenience, we present the auxiliary lemmas used in this paper.

\begin{lemma}[Symmetrization Lemma \cite{wainwright2019high}]\label{lem:symmetrization}
 Let $V$ be a class of functions on the domain $\Omega$, and let $\mP_\Omega$ is a given distribution on $\Omega$. Then
 \begin{equation*}
 \mathbb{E}\left[\sup_{v\in V} \left|\frac{1}{n} \sum_{j=1}^n v(\x_j) - \mathbb{E}_{\x\sim \mP_\Omega} v(\x) \right|\right]
 \leq 2 R_n(V).
 \end{equation*}
\end{lemma}

In addition, the following lemma shows the relationship between the covering number and the Rademacher complexity.
\begin{lemma}[Dudley’s integral theorem \cite{lecture2022note}] \label{thm:dudley} Let $V$ be a function class such that $\sup_{v\in V}\|v\|_{n,2}\leq M$. Then $\widehat{R}_n(V) $ satisfies
\begin{equation*}
\widehat{R}_n(V) \leq \inf_{0\leq \delta\leq M} \left\{4\delta + \frac{12}{\sqrt{n}}\int_\delta^M \sqrt{\log \mN(\epsilon, V, \|\cdot\|_{n,2})} \,d\epsilon\right\}.
\end{equation*}
\end{lemma}

\begin{lemma}[Talagrand’s concentration inequality \cite{bartlett2005local,rio2002inegalite}] \label{lem:talagrand_ineq}
 Let \( \varrho > 0 \), and let \( x_i \) be independent random variables distributed according to \( \mu \) and let \( V \) be a class of functions from \( \mathcal{X} \) to \( \R \). Assume that all functions \( v \) in \( V\) satisfy \( \mathbb{E}[v] = 0 \) and \( \|v\|_{\infty} \leq \varrho \). Let \( \zeta \) be a positive real number such that \( \zeta^2 \geq \sup_{v \in V} \mathrm{Var}[v(x_i)] \). Then, for any \( t \geq 0 \), with probability at least \( 1 - e^{-t} \),
\[
Z \leq \mathbb{E}[Z] + \sqrt{2tw} + \frac{\varrho t}{3},
\]
where \( Z = \sup_{v \in V} \sum_{i=1}^n v(x_i)\)  and \( w = n \zeta^2 + 2\varrho\mathbb{E}[Z] \).

\end{lemma}

\begin{lemma}[Ledoux-Talagrand contraction \cite{ledoux2013probability}]
 \label{lem:Ledoux-Talagrand contraction}
Let $ \{\tau_i\}_{i=1}^n $ be independent Rademacher random variables. Suppose the functions
$ \phi_i: \R \to \R,\ 1 \leq i \leq n $  are $L$-Lipschitz with $\phi_i(0) = 0$. Then for any bounded set $T\in \R^n$
\[
\mathbb{E}_\tau \left[ \sup_{(t_1, \dots, t_n) \in T} \left|\sum_{i = 1}^n\tau_i \phi_i(t_i)\right| \right] \leq 2L \cdot \mathbb{E}_\tau \left[ \sup_{(t_1, \dots, t_n) \in T} \left| \sum_{i = 1}^n \tau_i t_i \right| \right].
\]
\end{lemma}

\begin{lemma}[Lemma 2.6 \cite{ying2024accurate}] \label{lem:ying2024accurate}
Under the notation and conditions of \cref{lem:Ledoux-Talagrand contraction}, for arbitrary vector $\boldsymbol{c}=(c_1,\cdots,c_n)\in\R^n$, we have
\[
\mathbb{E}_\tau \left[ \sup_{(t_1, \dots, t_n) \in T} \left|\sum_{i = 1}^n\tau_i c_i \phi_i(t_i)\right| \right] \leq 2L \cdot \mathbb{E}_\tau \left[ \sup_{(t_1, \dots, t_n) \in T} \left|\sum_{i = 1}^n \tau_i c_i t_i\right| \right].
\]
\end{lemma}

\begin{lemma}[cf. \cite{bartlett2005local}]\label{lem: subroot_f_p1}
If $\phi: [0, \infty) \rightarrow [0, \infty)$ is a nontrivial sub-root function, then it is continuous on $[0, \infty)$, and the equation $\phi(r) = r$ has a unique positive solution. Moreover, if we denote the solution by $ r^* $, then for all $r > 0$, $r \geq \phi(r)$ if and only if $ r^*  \leq r$.
\end{lemma}

\begin{lemma}[cf. \cite{bartlett2005local}]\label{lem: subroot_f_p2}
Let $F$ be a class of functions (which may depend on the data), and let $T: F \rightarrow \R^{+}$ be a (possibly random) function that satisfies $T(\xi f) \leq \xi^2 T(f)$ for any $f \in F$ and any $\xi \in [0, 1]$. Then the (random) function $\phi$ defined for $r \geq 0$ by
$\phi(r) = \widehat{R}_n \left\{ f \in \operatorname{star}\left(F, f^*\right) : T\left(f - f^*\right) \leq r \right\}$ is sub-root, and $r \mapsto \mathbb{E} \phi(r)$ is also sub-root, where $\operatorname{star}(F, f^*)$ is the star hull of $F$ around $f^*$ and is defined by
$ \operatorname{star}(F, f^*) = \{ f^* + \xi(f-f^*): f\in F,\ \xi\in[0,1]\}.$
\end{lemma}

\begin{lemma}[cf. \cite{bartlett2005local}]
\label{lem: subroot_f_p3}
Let $F$ be a class of functions that map $\mathcal{X}$ into $[-b, b]$ with $b > 0$. For every $x > 0$ and $r$ satisfying
$$
r \geq 10b R_n \left\{ f: f \in F, \, \mathbb{E}[f^2] \leq r \right\} + \frac{11b^2 x}{n},
$$
then with probability at least $1 - e^{-x}$,
\[\left\{ f \in F:\, \mathbb{E}[f^2] \leq r \right\} \subset \left\{ f \in F: \|f\|_{n,2}^2 \leq 2r \right\}.
\]
\end{lemma}

\section[Proof of an a priori estimate]{Proof of the estimate \cref{eq:priori_estimate_rnn}}\label{proof:prior}
\begin{proof}
The interface problem \eqref{eq:interface_problem} can be recast as:
\begin{equation*}
\begin{aligned}
-\Delta \generalfunc_i(\x) &= g(\x)/\beta_i,\quad \text{in} \ \Omega_i, \ i=1,2, \\
 \llbracket \beta\nabla \generalfunc(\x)\cdot \mathbf{n} \rrbracket &= g_n(\x), \quad
\llbracket \generalfunc(\x)\rrbracket =g_d(\x),\quad \text{on} \ \Gamma,\\
\generalfunc(\x)&= g_b(\x),\quad \text{on} \ \partial\Omega.
\end{aligned}
\end{equation*}
Fix $c \in [0,1]$. We introduce two auxiliary functions {$\Tilde{\generalfunc}_i \in H^2(\Omega_i),$} $i = 1,2$, satisfying:
\begin{equation*}
\left\{
\begin{aligned}
 -\Delta \Tilde{\generalfunc}_1 &= g/\beta_1, \quad \text{in } \Omega_1, \\
 \Tilde{\generalfunc}_1 &= -c g_d, \quad \text{on } \Gamma,
\end{aligned}
\right. \quad
\left\{
\begin{aligned}
 -\Delta \Tilde{\generalfunc}_2 &= g/\beta_2, \quad \text{in } \Omega_2, \\
 \Tilde{\generalfunc}_2 &= (1-c) g_d, \quad \text{on } \Gamma, \\
 \Tilde{\generalfunc}_2 &= g_b, \quad \text{on } \partial\Omega.
\end{aligned}
\right.
\end{equation*}
According to \cite{agmon1959estimates}, the following regularity estimates hold:
\begin{equation}
 \begin{aligned} &\|\Tilde{\generalfunc}_1\|_{0,\Omega_1}\lesssim \norm{g}_{-2,\Omega_1}/\beta_1 + c \|g_d\|_{-1/2,\Gamma},\\ &\|\Tilde{\generalfunc}_2\|_{0,\Omega_2}\lesssim \norm{g}_{-2,\Omega_2}/\beta_2 + (1-c ) \|g_d\|_{-1/2,\Gamma} + \|g_b\|_{-1/2,\partial\Omega}.
 \end{aligned}
 \label{eq: tilde_u_estimates}
\end{equation}
Defining $\Tilde{g}(\x) = \beta_1\partial_{\mathbf{n}}\Tilde{\generalfunc}_1(\x)-\beta_2\partial_{\mathbf{n}}\Tilde{\generalfunc}_2(\x)$ on $\Gamma$, we have
\begin{equation}\label{eq:tildeg}
\|\tilde g\|_{-3/2,\Gamma}
\lesssim \beta_1\|\Tilde{\generalfunc}_1\|_{0,\Omega_1}+\beta_2\|\Tilde{\generalfunc}_2\|_{0,\Omega_2}.
\end{equation}

Let
$\Bar{\generalfunc}_i(\x)= \generalfunc_i(\x)-\Tilde{\generalfunc}_i(\x)$ with $\x\in \Omega_i$. Then $\Bar{\generalfunc}_i(\x)\in H^2(\Omega_i)$ and satisfies
\begin{equation*}
\left\{
 \begin{aligned}
 -\nabla\cdot(\beta \nabla \Bar \generalfunc) &= 0,\quad \text{in} \ \Omega,\\
 \llbracket \beta\nabla \Bar{\generalfunc}\cdot \mathbf{n} \rrbracket = g_n+\Tilde{g}, \quad
 \llbracket \Bar{\generalfunc}\rrbracket &=0, \quad \text{on} \ \Gamma,\\
 \Bar{\generalfunc} &= 0, \quad \text{on} \ \partial\Omega.
 \end{aligned}
 \right.
\end{equation*}
For every $w\in L^2(\Omega)$, by \cite{huang2002some,xu1998some} and the trace theorem, the adjoint
interface problem
\begin{equation*}
\left\{
\begin{aligned}
-\nabla\!\cdot(\beta_i\nabla \mu _i) &= w \quad &&\text{in }\Omega_i,\ i=1,2,\\
\llbracket \mu  \rrbracket &=0 \quad &&\text{on }\Gamma,\\
\llbracket \beta \partial_{\mathbf n} \mu  \rrbracket &=0 \quad &&\text{on }\Gamma,\\
\mu &=0 \quad &&\text{on }\partial\Omega,
\end{aligned}
\right.
\end{equation*}
admits a unique solution $\mu$ satisfying $\mu_i\in H^2(\Omega_i)$ ($i=1,2$) and
\begin{equation*}
\|\mu \|_{2,\Omega_1} \lesssim \frac{\|w\|_{0,\Omega_1}}{\beta_1}, \quad \|\mu \|_{2,\Omega_2} \lesssim  \frac{\|w\|_{0,\Omega_2}}{\beta_2}, \quad \|  \mu \|_{3/2,\Gamma}
\lesssim  T_{\beta}\|w\|_{0,\Omega},
\end{equation*}
where $T_{\beta} = 1/\max\{\beta_1, \beta_2\}$. Taking integration by parts, we have
\begin{equation*}
(\Bar \generalfunc,w)_{L^2(\Omega)} = \langle g_n+\Tilde{g}, \mu |_{\Gamma}\rangle_{\Gamma}
\qquad \forall\, w\in L^2(\Omega),
\end{equation*}
which yields
\begin{equation*}
|(\Bar \generalfunc,w)_{L^2(\Omega)}| \leq \norm{g_n+\Tilde{g}}_{-3/2,\Gamma} \norm{\mu }_{3/2,\Gamma} \lesssim T_{\beta} \norm{g_n+\Tilde{g}}_{-3/2,\Gamma}\|w\|_{0,\Omega} .
\end{equation*}
Therefore we have
\begin{equation*}
\|\Bar \generalfunc\|_{0, \Omega} \lesssim T_{\beta}\|g_n+\tilde g\|_{-3/2,\Gamma}.
\end{equation*}
Combining the last estimate with \eqref{eq: tilde_u_estimates}-\eqref{eq:tildeg} gives
\[
\begin{aligned}
 \|\generalfunc\|_{0,\Omega}
 \le \|\Bar{\generalfunc} \|_{0,\Omega} + \|\Tilde{\generalfunc}\|_{0,\Omega}   &\lesssim \left(\frac{1}{\beta_1}+T_\beta\right)\norm{g}_{0,\Omega_1} + \left(\frac{1}{\beta_2}+T_\beta\right)\norm{g}_{0,\Omega_2} + T_\beta\norm{g_n}_{0,\Gamma} \\ \quad &  + \left(1+T_\beta(c\beta_1+(1-c)\beta_2)\right)\norm{g_d}_{0,\Gamma
} + \left(1+T_\beta\beta_2\right)\norm{g_b}_{0,\partial\Omega}.
\end{aligned}
\]
A sharper bound follows by taking $c = 0$ if $\beta_1>\beta_2$, and $c = 1$ if $\beta_1< \beta_2$.
\end{proof}

\section[Proof of a covering number lemma]{Proof of \cref{lem:covering_number_tanh_nn}}
\label{proof:covering number}
\begin{proof}
Every $u\in G$ can be written as
\(
u(\x) = \sum_{m=1}^Ma_m \tanh(\vec{\omega}_m\x+ b_m),
\)
where parameters $(\vec{\omega}_m, b_m)\in [-q,q]^{d+1}$ are fixed. For each $i\in \{1,\cdots,M\}$, consider the following two function classes:
\[
\begin{aligned}
 &G_1^{(i)} = \left\{a_i\tanh(\vec{\omega}_i\x+ b_i):\ \sum_{m=1}^M(1+\norm{\vec{\omega}_{m}}_1)^2|a_m|\leq C_B \right\}, \\
 &G_2^{(i)} =\left\{a_i\tanh(\vec{\omega}_i\x+ b_i):\ |a_i|\leq C_B \right\}.
\end{aligned}
\]
Clearly $G_1^{(i)}  \subset  G_2^{(i)}$. Since $G_2^{(i)}$ is parameterized by the single scalar $a_i$, we have
\[
\mathcal{N}\left(\epsilon, G_1^{(i)}, \|\cdot\|_{n,1}\right) \leq \mathcal{N}(\epsilon, G_2^{(i)}, \|\cdot\|_{n,1}) \leq \frac{2C_B}{\epsilon}, \quad \epsilon>0.
\]
Viewing $G$ as the sum of the $M$ subclasses $G_1^{(i)}$ and applying Lemma 16.4 of \cite{gyorfi2006distribution}, we obtain
\begin{equation}
\label{eq:c_num_f_0}
 \mathcal{N}\left(\epsilon, G, \|\cdot\|_{n,1}\right) \leq \left(\frac{2C_BM}{\epsilon}\right)^{M}.
\end{equation}
Next, since
$\tanh(x)$ and its first and second derivatives are uniformly bounded and $\norm{\vec{\omega}_m}_{\infty}\leq q$, a similar argument yields
\begin{equation}
 \label{eq:c_num_f_1_2}
 \mathcal{N}\left(\epsilon, D^{\vec{k}}G, \|\cdot\|_{n,1}\right) \leq \left(\frac{2q C_BM}{\epsilon}\right)^{M},\ \mathcal{N}\left(\epsilon, D^{\vec{k}}G, \|\cdot\|_{n,1}\right) \leq \left(\frac{2q^2 C_BM}{\epsilon}\right)^{M}
\end{equation}
for $|\vec{k}|=1$ and $|\vec{k}|=2$, respectively. Combining \eqref{eq:c_num_f_0} and \eqref{eq:c_num_f_1_2} yields the conclusion.
\end{proof}

\end{document}